\documentclass{amsart}
\usepackage{amssymb}
\usepackage{latexsym}
\usepackage{amsmath}
\usepackage{enumerate}

\newtheorem{theorem}{Theorem}[section]
\newtheorem{lemma}[theorem]{Lemma}
\newtheorem{corollary}[theorem]{Corollary}
\newtheorem{proposition}[theorem]{Proposition}

\theoremstyle{definition}
\newtheorem{definition}[theorem]{Definition}

\newtheorem{fact}[theorem]{Fact}

\newtheorem{claim}[theorem]{Claim}

\newtheorem{the main theorem}[theorem]{The Main Theorem}

\newtheorem*{notation}{Notation}
\newtheorem*{acknowledgement}{Acknowledgement}

\numberwithin{equation}{section}

\begin{document}
\title[]{Forcing a set model of $Z_3\, +$ Harrington's Principle}

\author{Yong Cheng}
\address{Institut f$\ddot{u}$r mathematische Logik und Grundlagenforschung, Fachbereich Mathematik und Informatik, Universit$\ddot{a}$t M$\ddot{u}$nster, Einsteinstr. 62, 48149  M$\ddot{u}$nster, Germany}

\email{world-cyr@hotmail.com}

\keywords{Harrington's Principle, strong reflecting property, remarkable cardinal,  Baumgartner's forcing.}
\subjclass[msc2010]{03E35, 03E55, 03E30}



\begin{abstract}
Let $Z_3$ denote $3^{rd}$ order arithmetic. Let Harrington's Principle, {\sf HP}, denote the statement that there is a real $x$ such that every $x$--admissible ordinal is a cardinal in $L$. In this paper, assuming  there exists a remarkable cardinal with a weakly inaccessible cardinal above  it, we force a set model of $Z_3\, + \, {\sf HP}$ via set forcing without reshaping.
\end{abstract}

\maketitle


\section{Introduction}

Harrington proved in 1978 the following classical theorem which stimulates the research on the relationship between large cardinals and  determinacy hypothesis since then.

\begin{theorem}\label{Martin-Harrinton theorem}
(Harrington, \cite{Harrington1}) \quad $(ZF) \quad  Det(\Sigma_1^1)$ implies  $0^{\sharp}$ exists.
\end{theorem}

\begin{definition}
Let Harrington's Principle, {\sf HP} for short, denote the following statement: $\exists x\in 2^{\omega}\forall \alpha(\alpha$ is $x$-admissible $\rightarrow \alpha$ is an $L$-cardinal).
\end{definition}

\begin{theorem}\label{Silver's theorem}
(Silver, \cite{Harrington1}) \quad (ZF) \quad {\sf HP} implies $0^{\sharp}$ exists.
\end{theorem}

\begin{definition}
\begin{enumerate}[(i)]
  \item $Z_{2}= ZFC^{-} +$ Any set is Countable.\footnote{$ZFC^{-}$ denotes $ZFC$  with the Power Set Axiom deleted and Collection instead of Replacement.}
  \item $Z_{3}= ZFC^{-} + \mathcal{P}(\omega)$ exists + Any set is of cardinality $\leq \beth_1$.
      \item $Z_4=ZFC^{-}+ \mathcal{P}(\mathcal{P}(\omega))$ exists + Any set is of cardinality $\leq \beth_2$.
\end{enumerate}
\end{definition}

$Z_2, Z_3$ and $Z_4$ are the corresponding axiomatic systems for Second Order Arithmetic (SOA), Third Order Arithmetic and Fourth Order Arithmetic. Note that $Z_3\vdash H_{\omega_1}\models Z_2, Z_4\vdash H_{\beth_1^{+}}\models Z_3$ and $``\exists A\subseteq\omega_1(V=L[A]) \, + \, Z_3"\vdash\omega_1$ is the largest cardinal.

The known proofs of Theorem \ref{Martin-Harrinton theorem} are done in two steps: first show that $Det(\Sigma_1^1)$ implies  {\sf HP} and then show that {\sf HP} implies $0^{\sharp}$ exists. We observe that the first step is provable in  $Z_2$. For the proof of $``Z_2\, +\, Det(\Sigma_1^1)$ implies  {\sf HP}", see \cite{YongCheng}. In this paper, we aim to prove the following main theorem.

\begin{the main theorem}\label{the main thm of the thesis}
(Set forcing) \quad Assuming there exists a remarkable cardinal with a weakly inaccessible cardinal above it, we can force a model of $Z_3\, + \, {\sf HP}$.
\end{the main theorem}

As a corollary, $Z_3\, + \,  {\sf HP}$ does not imply $0^{\sharp}$ exists. But $Z_4\, + \, {\sf HP}$ implies $0^{\sharp}$ exists which we construe as part of the folklore, cf.\cite{Harrington1}. So $Z_4$ is the minimal system in higher order arithmetic to show that {\sf HP} implies  $0^{\sharp}$ exists.
The Main Theorem \ref{the main thm of the thesis} is proved via set forcing and we do not use the reshaping technique. 

The history of the main result in this paper is as follows: The theorem ``$Z_3\, + \,  {\sf HP}$ does not imply $0^{\sharp}$ exists" was first proved in \cite{YongCheng}. Results in \cite{YongCheng} are proved via set forcing and we do not use the reshaping technique. However, the large cardinal strength of ``$Z_3\, + \,  {\sf HP}$" is not discussed in \cite{YongCheng}. In latter joint work with Ralf Schindler in \cite{jointpaperwithralf}, we compute the exact  large cardinal strength of ``$Z_3\, + \,  {\sf HP}$". Results in  \cite{jointpaperwithralf} are proved via class forcing. In \cite[Theorem 3.2]{jointpaperwithralf}, assuming there is one remarkable cardinal, we force via class forcing a class model of $Z_3\, + \,  {\sf HP}$ using the reshaping technique. The proof of The Main Theorem \ref{the main thm of the thesis} in this paper is based on \cite{YongCheng} and we improve the presentation in \cite{YongCheng} by computing the upper bound of the large cardinal hypothesis used in Step One in Section \ref{first secton step one} via the notion of remarkable cardinal which is much weaker than the large cardinal hypothesis used in \cite{YongCheng}. 
 

\section{Definitions and preliminaries}\label{forcing background}

Our definitions and notations are standard. We refer to standard textbooks as \cite{Jech}, \cite{Higherinfinite} and \cite{Kunen1} for the definitions and notations we use. For the definition of admissible set and admissible ordinal, see \cite{Barwise} and \cite{Constructiblity}. For notions of large cardinals, see \cite{Higherinfinite}. Our notations about forcing are standard (see \cite{Jech2} and \cite{Jech}). Almost disjoint forcing is standard(see \cite{Jech} and \cite{Kunen1}). We say that $0^{\sharp}$ exists if there exists an iterable premouse of the form $(J_{\alpha}, \in, U)$ where $U\neq\emptyset$. For the  theory of $0^{\sharp}$ see e.g.\
\cite{schindler}. We can define $0^{\sharp}$ in $Z_2$.  In $Z_2, 0^{\sharp}$ exists if and only if $\exists x\in \omega^{\omega}(x$ codes a countable iterable premouse) which is a $\Sigma^1_3$ statement.

Note that under $V=L, H_{\eta}=L_{\eta}$ for any $L$-cardinal $\eta$. In this paper, we often use $H_{\eta}$ and $L_{\eta}$ interchangeably. Throughout this paper whenever we write $X\prec H_{\kappa}$ and $\gamma\in X$, $\bar{\gamma}$ always denotes the image of $\gamma$ under the transitive collapse of $X$.

\begin{definition}
(Ralf Schindler, \cite{Schindler2})
\begin{enumerate}[(i)]
  \item A cardinal $\kappa$ is remarkable if and only if for all regular cardinal $\theta>\kappa$ there are $\pi, M, \bar{\kappa}, \sigma, N$ and $\bar{\theta}$ such that the following hold: $\pi: M\rightarrow H_{\theta}$ is an elementary embedding, $M$ is countable and transitive, $\pi(\bar{\kappa})=\kappa$, $\sigma: M\rightarrow N$ is an elementary embedding with critical point $\bar{\kappa}$, $N$ is countable and transitive, $\bar{\theta}=M\cap Ord$ is a regular cardinal in $N, \sigma(\bar{\kappa})>\bar{\theta}$ and $M=H_{\bar{\theta}}^{N}$, i.e. $M\in N$ and $N\models M$ is the set of all sets which are hereditarily smaller than $\bar{\theta}$.
  \item Let $\kappa$ be a cardinal, $G$ be $Col(\omega, <\kappa)$-generic over $V$, $\theta>\kappa$ be a regular cardinal and $X\in [H_{\theta}^{V[G]}]^{\omega}$. We say that $X$ condense remarkably if $X=ran(\pi)$ for some elementary $\pi: (H_{\beta}^{V[G\cap H_{\alpha}^{V}]}, \in, H_{\beta}^{V}, G\cap H_{\alpha}^{V})\rightarrow (H_{\theta}^{V[G]}, \in, H_{\theta}^{V}, G)$ where $\alpha=crit(\pi)<\beta<\kappa$ and $\beta$ is a regular cardinal in $V$.
\end{enumerate}
\end{definition}

\begin{lemma}\label{key lemma on remarkable cn}
(Ralf Schindler, \cite{Schindler2}) \quad A cardinal $\kappa$ is remarkable if and only if  for all regular cardinal $\theta>\kappa$ we have $\Vdash^{V}_{Col(\omega, <\kappa)} ``\{X\in [H_{\check{\theta}}^{V[\dot{G}]}]^{\omega}: X$  condense remarkably\} is stationary''.
\end{lemma}

\begin{lemma}\label{key lemma on relationship between remark and wfp}
Suppose $\kappa$ is an $L$-cardinal. The followings are equivalent:
\begin{enumerate}[(a)]
  \item $\kappa$ is remarkable in $L$;
  \item  If $\gamma\geq\kappa$ is an $L$-cardinal, $\theta>\gamma$ is  a regular cardinal in $L$, then $\Vdash^{L}_{Col(\omega, <\kappa)} ``\{X| X\prec L_{\check{\theta}}[\dot{G}], |X|=\omega\wedge \check{\gamma}\in X\wedge \bar{\check{\gamma}}$ is an $L$-cardinal\} is stationary".
\end{enumerate}
\end{lemma}
\begin{proof}
By Lemma \ref{key lemma on remarkable cn}, $\kappa$ is remarkable in $L$ \text{if{f}} if $\theta>\kappa$ is a regular cardinal in $L$ and $G$ is $Col(\omega, <\kappa)$-generic over $L$, then $L[G]\models ``\{X\in [L_{\theta}[G]]^{\omega}| X=ran(\pi)$ for some elementary $\pi: (L_{\beta}[G\upharpoonright \alpha], \in, L_{\beta}, G\upharpoonright\alpha)\rightarrow (L_{\theta}[G], \in, L_{\theta}, G)$ where $\alpha=crit(\pi)<\beta<\kappa$ and $\beta$ is a regular cardinal in $L$\} is stationary'' \text{if{f}} if $\gamma\geq\kappa$ is an $L$-cardinal, $\theta>\gamma$ is  a regular cardinal in $L$ and $G$ is $Col(\omega, <\kappa)$-generic over $L$, then $L[G]\models ``\{X| X\prec L_{\theta}\wedge |X|=\omega\wedge \gamma\in X\wedge \bar{\gamma}$ is an $L$-cardinal\} is stationary".
\end{proof}

In the rest of this section, we assume that $S$ is a  stationary subset of $\omega_1$.

\begin{definition}
(Harrington's forcing, \cite{Harrington2})\quad $P_{S}=\{p: p$ is a closed bounded subset of $\omega_1$ and $p\subseteq S\}$. For $p,q\in P_{S}, p\leq q$ if and only if $p\supseteq q$ and for any $\alpha\in p\setminus q, \alpha> sup(q)$.\footnote{$|P_{S}|=2^{\omega}, P_{S}$ is $\omega$-distributive and hence assuming $CH$,  $P_{S}$ preserves all cardinals.}
\end{definition}

\begin{definition}
(Baumgartner's forcing, \cite{Baumgartner})\quad Define $P_{S}^{B}=\{f: dom(f)\rightarrow S\mid dom(f)\subseteq\omega_1$ is finite and $\exists\alpha>\max(dom(f))\, \exists g: \alpha\rightarrow S(g$ is continuous,  increasing and $g\upharpoonright dom(f)=f$)\}. For $f, g\in P_{S}^{B}, g\leq f$ if and only if $f\subseteq g$.
\end{definition}

Note that the following are equivalent: (1) $f\in P_{S}^{B}$; (2) $dom(f)\subseteq\omega_1$ is finite and there exists $g: \max(dom(f))+1\rightarrow S$ such that $g$ is continuous, increasing and $g\upharpoonright dom(f)=f$; (3) $dom(f)\subseteq\omega_1$ is finite and there exists $C\subseteq S$ such that $C$ is closed, $o.t.(C)=\max(dom(f))+1$ and for any $\beta\in dom(f), f(\beta)$ is the $\beta$-th element of $C$.

Let $G$ be $P_{S}^{B}$-generic over $V$. Define $F_{G}=\bigcup\{f\mid f\in G\}$. Then $F_{G}: \omega_1\rightarrow S$ is increasing, continuous and $ran(F_{G})$ is a club in $\omega_1$.

\begin{fact}\label{Baumgartner fact on forcing notion}
(Baumgartner, \cite{Baumgartner})\quad ($Z_3$)\quad $|P_{S}^{B}|=\omega_1$ even not assuming $CH$ and $P_{S}^{B}$ preserves $\omega_1$.
\end{fact}

Since $P_{S}^{B}$ is $\omega_2$-$c.c$ and preserves $\omega_1$, $P_{S}^{B}$ preserves all cardinals.

\begin{proposition}\label{reflecting}
Suppose $\gamma\geq\omega_1$ is an $L$-cardinal. Then the  following are equivalent:
\begin{enumerate}[(a)]
  \item For some  regular cardinal $\kappa> \gamma$, $\forall X((X\prec H_{\kappa}, |X|=\omega$ and $\gamma\in X)\rightarrow \bar{\gamma}$ is an $L$-cardinal).
  \item  There exists $F: \gamma^{<\omega}\rightarrow\gamma$ such that if   $X\subseteq\gamma$ is countable and closed under $F$, then $o.t.(X)$ is an $L$-cardinal.\footnote{In this paper we say $X$ is closed under $F$ if $F``X^{<\omega}\subseteq X$.}
  \item For any  regular cardinal $\kappa> \gamma$, $\forall X((X\prec H_{\kappa}\wedge |X|=\omega\wedge \gamma\in X)\rightarrow \bar{\gamma}$ is an $L$-cardinal).
\end{enumerate}
\end{proposition}
\begin{proof}
$(a)\Rightarrow (b)$ Let $\kappa>\gamma$ be the witness regular cardinal for $(1)$. Let $Z=\{X\mid X\prec H_{\kappa},  |X|=\omega$ , $\gamma\in X$ and $\bar{\gamma}$ is an $L$-cardinal\}. Then $Z\upharpoonright\gamma=\{X\cap\gamma\mid X\in Z\}$ contains a club $E$ in $[\gamma]^{\omega}$. So there exists $F: \gamma^{<\omega}\rightarrow\gamma$ such that if $X\subseteq\gamma$ is countable and closed under $F$, then $X\in E$. Suppose $X\subseteq\gamma$ is countable and closed under $F$. We show that $o.t.(X)$ is an $L$-cardinal. Since  $X\in E$, $X=Y\cap\gamma$ for some $Y\in Z$ and hence $\bar{\gamma}$ is an $L$-cardinal. So $o.t.(X)=o.t.(Y\cap\gamma)=\bar{\gamma}$  is  an $L$-cardinal.

$(b)\Rightarrow (c)$ Suppose $\kappa>\gamma$ is regular, $X\prec H_{\kappa}$, $|X|=\omega$ and  $\gamma\in X$.  We show that $\bar{\gamma}$  is an $L$-cardinal. By (b), take $ F\in X$ such that in $X, F: \gamma^{<\omega}\rightarrow\gamma$ has the property that
\begin{equation}\label{property of function}
\text{if $X\subseteq\gamma$ is countable and closed under $F$, then $o.t.(X)$ is an $L$-cardinal.}
\end{equation}
Since $X\cap\gamma$ is closed under $F$, by (\ref{property of function}), $o.t.(X\cap\gamma)$ is an $L$-cardinal. But $\bar{\gamma}=o.t.(X\cap\gamma)$.
\end{proof}

\begin{definition}\label{definition of strong reflecting property and weakly}
Let $\gamma$ be an $L$-cardinal. If $\gamma\geq\omega_1$, we say $\gamma$ has the strong reflecting property if Proposition \ref{reflecting}(a) holds. If $\gamma<\omega_1$, we say that $\gamma$ has the strong reflecting property iff $\gamma=\gamma$. 
\end{definition}

\begin{proposition}\label{equivalent forms of strong reflecting property}
Suppose $\gamma\geq\omega_1$ is an $L$-cardinal  and $|\gamma|=\omega_1$. Then the  following are equivalent:
\begin{enumerate}[(a)]
\item $\gamma$ has the strong reflecting property.
  \item For any bijection $\pi: \omega_1\rightarrow \gamma$, there exists a club $D\subseteq\omega_1$ such that for any $\theta\in D$, $o.t.(\{\pi(\alpha)\mid\alpha< \theta\})$ is an $L$-cardinal.
      \item For some bijection $\pi: \omega_1\rightarrow \gamma$, there exists a club $D\subseteq\omega_1$ such that for any $\theta\in D$, $o.t.(\{\pi(\alpha)\mid\alpha< \theta\})$ is an $L$-cardinal.
\end{enumerate}
\end{proposition}
\begin{proof}
$(a)\Rightarrow (b)$ Let $\kappa>\gamma$ be the regular cardinal that witnesses the  strong reflecting property of $\gamma$. Suppose $\pi: \omega_1\rightarrow \gamma$ is a bijection. Let $E=\{X\cap \omega_1\mid X\prec H_{\kappa}, |X|=\omega, \pi\in X$ and $\gamma\in X\}$. Then $E$ contains a club $D$ in $\omega_1$. Let $\beta\in D$.  Then $\beta=X\cap \omega_1$ for some $X$ such that $\pi\in X, X\prec H_{\kappa}, |X|=\omega$ and $\gamma\in X$. Note that $\bar{\gamma}=o.t.(\{\pi(\alpha)\mid\alpha< X\cap \omega_1\})$. So $o.t.(\{\pi(\alpha)\mid\alpha< \beta\})=\bar{\gamma}$ is an $L$-cardinal.

$(c)\Rightarrow (a)$ Let $\kappa>\gamma$ be a regular cardinal with $\kappa\geq(2^{\omega_1})^{+}$. Suppose $X\prec H_{\kappa},|X|=\omega$ and $\gamma\in X$. We show that $\bar{\gamma}$ is an $L$-cardinal. By $(c)$, take $\pi, D\in X$ such that $\pi: \omega_1\rightarrow\gamma$ is a bijection and $D\subseteq\omega_1$ is the witness club for $\pi$ in $(c)$. Since $D$ is unbounded in $X\cap\omega_1, X\cap\omega_1\in D$. Note that $\bar{\gamma}=o.t.(\{\pi(\alpha)\mid\alpha\in X\cap \omega_1\})$. So $\bar{\gamma}$ is an $L$-cardinal.
\end{proof}

Let $(i)^{\ast}, (ii)^{\ast}$ and $(iii)^{\ast}$ respectively denote the statements which replace ``is an $L$-cardinal" with ``is not an $L$-cardinal" in  Proposition \ref{reflecting}(a), Proposition \ref{reflecting}(c) and Proposition \ref{equivalent forms of strong reflecting property}(b). The following corollary is an observation from proofs of Proposition \ref{reflecting} and Proposition \ref{equivalent forms of strong reflecting property}.

\begin{corollary}\label{corollary about strong reflecting property}
Suppose $\gamma\geq\omega_1$ is an $L$-cardinal  and $|\gamma|=\omega_1$. Then $(ii)^{\ast} \Leftrightarrow (i)^{\ast}\Leftrightarrow (iii)^{\ast}$.
\end{corollary}

\begin{proposition}\label{upward abso for reflecting proper}
Suppose $\gamma\geq\omega_1$ is an $L$-cardinal. The statement $``\gamma$ has the strong reflecting property'' is upward absolute.
\end{proposition}
\begin{proof}
Suppose $M\subseteq N$ are inner models and $M\models \gamma\geq\omega_1$ has the strong reflecting property. We show that $N\models\gamma$ has the strong reflecting property.

By Proposition \ref{reflecting}, in $M$, there exists $F: \gamma^{<\omega}\rightarrow\gamma$ such that (\ref{property of function}) holds. If $\gamma$ is countable in $N$, by definition, $\gamma$ has the strong reflecting property in $N$. Assume that $N\models \gamma$ is uncountable. By Proposition \ref{reflecting}, it suffices to show that in $N$, (\ref{property of function}) holds.

Suppose not. Then in $N$, there  exists $\bar{\gamma}<\omega_1$ such that $\bar{\gamma}$ is not an $L$-cardinal and there exists an order preserving $j: \bar{\gamma}\rightarrow\gamma$ such that $ran(j)$ is closed under $F$. So in $N$, there  exists $e: \omega\rightarrow L_{\omega_1^{N}}$ and $\gamma^{\prime}\in e``\omega$ such that $e``\omega\prec L_{\omega_1^{N}}, L_{\omega_1^{N}}\models ``\gamma^{\prime}$ is not an $L$-cardinal" and there exists an order preserving $j^{\prime}: o.t.(e``\omega\cap \gamma^{\prime})\rightarrow\gamma$ such that $ran(j^{\prime})$ is closed under $F$.

Let $\langle \varphi_i\mid i\in\omega\rangle$ be a recursive enumeration of formulas with infinite repetitions. We assume that for  $i\in\omega, \varphi_i$ has free variables among $x_0, \cdots, x_{i+1}$. So in $N$, there exist $e: \omega\rightarrow L_{\omega_1^{N}}, \pi: \omega\rightarrow\gamma$ and $\gamma^{\ast}\in e``\omega$ such that $(i)$ for any $i\in\omega$, if there exists $a\in L_{\omega_1^{N}}$ such that $L_{\omega_1^{N}}\models \varphi_i[a, e(0), \cdots, e(i)  ]$, then $L_{\omega_1^{N}}\models \varphi_i[e(2i+1), e(0), \cdots, e(i)]$; $(ii)$  $ran(\pi)$ is closed under $F$;  $(iii)$ $L_{\omega_1^{N}}\models \gamma^{\ast}$ is not an $L$-cardinal;  and $(iv)$ for  $i\in\omega$, if $e(i)\notin \gamma^{\ast}$, then $\pi(i)=0$;  for  $i< j\in\omega$, if $e(i), e(j)\in \gamma^{\ast}$, then $\pi(i)<\pi(j)\Leftrightarrow e(i)<e(j)$ and $\pi(i)=\pi(j)\Leftrightarrow e(i)=e(j)$. In $N$, let $T=\{(e\upharpoonright n, \pi\upharpoonright n)  :  e$ and $\pi$  have properties $(i)-(iv)$\}. $T$ is a tree and from 
$(i)-(iv)$, by absoluteness, $T\in M$. Since in $N$, there exists  $(e, \pi)$  satisfying $(i)-(iv)$, $T$ has an infinite branch in $N$. By absoluteness, $T$ has an infinite branch in $M$ and such a branch corresponds to the existence of $(e, \pi)$ with properties $(i)-(iv)$ in $M$. So in $M$, there exists $X\subseteq\gamma$ such that $X$ is countable, closed under $F$ and $o.t.(X)$ is not an $L$-cardinal which contradicts (\ref{property of function}).
\end{proof}

\section{Proof of The Main Theorem}\label{section for the main result}

In this section we prove The Main Theorem \ref{the main thm of the thesis}. Assuming there exists a remarkable cardinal with a weakly inaccessible cardinal above it, we force a set model of $Z_3\, + \, {\sf HP}$ via set forcing. We give an outline of our proof in Section \ref{final section}.  

\subsection{Step One}\label{first secton step one} 

In this step we force over $L$ to get a club in $\omega_2$ of $L$-cardinals with the strong reflecting property.

We work in $L$.  Let $\kappa$ be a remarkable cardinal and $\lambda>\kappa$ be an inaccessible cardinal.  Suppose $\bar{G}$ is $Col(\omega,<\kappa)$-generic over $L$ and $G$ is  $Col(\omega,<\kappa)\ast Col(\kappa, <\lambda)$-generic over $L$. Now we work in $L[G]$.
\[\text{Define $K=\{\gamma\mid \omega_1\leq\gamma<\omega_2$ and $\gamma$ is an $L$-cardinal\}}.\]
\begin{definition}
For $\gamma\in K$, we say $\gamma$ has the weakly reflecting property if for some bijection $\pi: \omega_1\rightarrow \gamma$, there exists  stationary $D\subseteq\omega_1$ such that for any $\theta\in D$, $o.t.(\{\pi(\alpha)\mid\alpha< \theta\})$ is an $L$-cardinal.
\end{definition}
\begin{proposition}\label{weakly re pro}
$L[G]\models$ for any $\gamma\in K, \gamma$ has the weakly reflecting property.
\end{proposition}
\begin{proof}
We work in $L[G]$. Suppose $\gamma\in K$ is a counterexample and $\theta>\gamma$ is a regular cardinal. Since $\kappa$ is remarkable in $L$, by Lemma \ref{key lemma on relationship between remark and wfp}, $L[\bar{G}]\models \{X| X\prec H_{\theta}\wedge|X|=\omega\wedge \gamma\in X\wedge \bar{\gamma}$ is an $L$-cardinal\} is stationary. Note that the property ``$X\prec  H_{\theta}\wedge |X|=\omega\wedge \gamma\in X \wedge\bar{\gamma}$ is an $L$-cardinal'' is absolute between $L[\bar{G}]$ and $L[G]$. So by absoluteness, in $L[G]$,
\begin{equation}\label{first equ in section one}
\text{$\exists X(X\prec  H_{\theta}\wedge |X|=\omega\wedge \gamma\in X \wedge\bar{\gamma}$ is an $L$-cardinal).}
\end{equation}
Since $\gamma$ does not have the weakly reflecting property, $(iii)^{\ast}$ in Corollary \ref{corollary about strong reflecting property} holds and hence, by Corollary \ref{corollary about strong reflecting property}, $(ii)^{\ast}$ holds which contradicts (\ref{first equ in section one}).
\end{proof}

So $K$ is a club in $\omega_2$ of $L$-cardinals with the weakly reflecting property. For $\gamma\in K$, by Proposition \ref{weakly re pro}, there exist a bijection $\pi: \omega_1\leftrightarrow \gamma$ and a stationary set $S\subseteq \omega_1$ such that for any $\theta\in S, o.t.(\{\pi(\alpha)\mid \alpha<\theta\}) $ is an $L$-cardinal(let $\pi_{\gamma}$ and $S_{\gamma}$ be such $\pi$ and $S$). Then $S_{\gamma}$ is stationary for $\gamma\in K$.

\begin{definition}
Suppose $\kappa $ is a regular cardinal and $\{P_i: i\in I\}$ is a collection of partial orders. The $\kappa$-product of $\{P_i: i\in I\}$ is defined as $P=\{p: dom(p)=I\wedge\forall i\in I (p(i)\in P_i)\wedge|suppt(p)|<\kappa\}$ where $suppt(p)=\{i\in I : p(i)\neq 1_{P_i}\}$.
\end{definition}

Let $P$ be the $\omega_1$-product of $\{P_{\gamma}: \gamma\in K\}$ where $P_{\gamma}$ is the Harrington forcing to shoot a club through $S_{\gamma}$. Since $CH$ holds in $L[G], |P_{\gamma}|=\omega_1$ for $\gamma\in K$.
\begin{fact}\label{fact on product forcing}
(\cite{Jech})\quad Assume $\kappa^{<\kappa}=\kappa$. If for every $i\in I, |P_i|\leq\kappa$, then the $\kappa$-product of $P_i$ satisfies $\kappa^{+}$-c.c.
\end{fact}

In $L[G], \omega_1^{<\omega_1}=\omega_1$. By Fact \ref{fact on product forcing}, $P$ is $\omega_2$-c.c. For  $\gamma\in K, P_{\gamma}$ is $\omega$-distributive and hence preserves $\omega_1$. The proof of the following lemma imitates Lemma 2.4 in \cite{schindler3}. 

\begin{lemma}
$P$ is $\omega$-distributive.
\end{lemma}
\begin{proof}
For $\gamma\in K$, we may view $P_{\gamma}$ as the set of all strictly increasing and continuous sequences $(\eta_{i}: i\leq\alpha)$ of countable successor length consisting of elements of $S_{\gamma}$.  For $p\in P$, we may write $p=\lbrace (\eta^{\lambda}_{i}(p): i\leq\alpha_{\lambda}(p)): \lambda\in suppt(p)\rbrace$. Let $\overrightarrow{D}=(D_n: n\in \omega)$ be a sequence of dense open subsets of $P$. Let $p\in P$. Pick some $Y\prec H_{\omega_3}$ such that $\omega_1\cup \lbrace p, P, \overrightarrow{D}\rbrace\subseteq Y, Y\cap \omega_2<\omega_2$ and $Y$ is of cardinality $\omega_1$. Let $\gamma=Y\cap \omega_2$. Then $\gamma$ is an $L$-cardinal and $\gamma\in K$. Since $S_{\gamma}$ is stationary, we may pick some countable $X\prec H_{\omega_3}$ such that $\lbrace p, P, \overrightarrow{D}, Y,\gamma\rbrace\subseteq X$ and $X\cap \gamma\in S_{\gamma}$. Then we have $\lbrace p, P, \overrightarrow{D}\rbrace\subseteq X\cap Y\prec Y\prec H_{\omega_3}$. We may therefore build a descending sequence $(p_{n}: n\in\omega)$ of conditions from $P$ such that $p_0=p, \lbrace p_{n}: n\in\omega\rbrace\subseteq X\cap Y, p_{n+1}\in D_n$ and for every $L$-cardinal $\lambda\in X\cap\gamma$ and every $\beta\in X\cap\gamma$ there is some $n\in\omega$ such that $\lambda\in suppt(p_n)$ and $\beta\in \eta^{\lambda}_{i}(p_n)$ for some $i\leq\alpha_{\lambda}(p_n)$. Let us write $\alpha=X\cap\omega_1$ and $q=\lbrace (\eta^{\lambda}_{i}: i\leq\alpha): \lambda\in  X\cap\gamma$ is an $L$-cardinal $\rbrace$ where for every $L$-cardinal $\lambda\in  X\cap\gamma$, if $i<\alpha$, then $\eta^{\lambda}_{i}=\eta^{\lambda}_{i}(p_n)$ for some (all) sufficiently large $n$ and $\eta^{\lambda}_{\alpha}=X\cap\lambda$. It is not hard to check that $q\in P, q\leq_{P} p$ and $q\in D_n$ for all $n\in\omega$.  
\end{proof}

So $P$ preserves $\omega_1$ and hence $P$ preserves all cardinals. Let $H$ be $P$-generic over $L[G]$. Now we work in $L[G, H]$. By $(a)\Leftrightarrow (c)$ in Proposition \ref{equivalent forms of strong reflecting property},
\begin{equation}\label{any elem of K has srp}
\text{$L[G, H]\models$ Any $\alpha\in K$ has  the strong reflecting property.}
\end{equation}

So $K$ is a club in $\omega_2$ of $L$-cardinals with the strong reflecting property.

\subsection{Step Two}\label{same proof section}

In this step, we work in $L[G, H]$ to find some $B_0\subseteq \omega_2$ and  $A\subseteq\omega_1$ such that $L[B_0, A]\models$ ``if $\omega_1\leq\alpha<\alpha_{A}$ is $A$-admissible, then $\alpha$ is an $L$-cardinal with the strong reflecting property" where $\alpha_{A}$ is the least $\alpha$ defined in $L[B_0,A]$ such that $L_{\alpha}[A]\models Z_3$. Then we define a stationary set $S$ and then show that $S$ contains a club.

We still work in $L[G, H]$. Note that $GCH$ holds. Let $(B_0,\gamma^{\ast})$ be such that (a) $\omega_1 < \gamma^{\ast} \leq \omega_2$, (b) $B_0 \subset \gamma^{\ast}$ and $\gamma^{\ast} = (\omega_2)^{L[B_0]}$, (c)  $L_{\gamma^{\ast}}[B_0] \prec L_{\omega_2}[G, H]$ and (d) $\gamma^{\ast}$ is as small as possible. Let $B$ be the theory of $(L_{\gamma^{\ast}}[B_0], B_0)$ with parameters from $\gamma^{\ast}$.\footnote{We define $(B_0,\gamma^{\ast})$ and $B$ in this way so that we can prove Claim \ref{key claim by Woodin}. The proof of Claim \ref{key claim by Woodin} makes full use of our definition of $(B_0,\gamma^{\ast})$ and $B$.} i.e. $B$ denotes the subset of $\gamma^{\ast}$ coded by $T$ where $T$ is the set of pairs $(e,s)$ where $e$ is the G$\ddot{o}$del number of a formula $\phi(x_0,\cdots, x_n), s$ is a sequence $(\alpha_0, \cdots, \alpha_n)$ of ordinals $<\gamma^{\ast}$ and $\phi[\alpha_0, \cdots,\alpha_n]$ holds in $(L_{\gamma^{\ast}}[B_0], B_0)$. 

We work in $L[B_0]$. To define an almost disjoint sequence $\langle\delta_{\beta}^{\ast}\mid \beta<\omega_2\rangle$ on $\omega_1$, we first define a sequence $\langle\sigma_{\beta}^{\ast}\mid \beta<\omega_2\rangle$ such that for each $\beta, \sigma_{\beta}^{\ast}$ is the $L[B_0]$-least $\sigma \subset \omega_1$ such that $\sigma$ has cardinality $\omega_1$ and $\sigma$ is different from $\sigma_{\alpha}^{\ast}$ for any $\alpha<\beta$. Let $\langle s_\alpha\mid \alpha\in\omega_1\rangle\in L[B_0]$ be an $<_{L[B_0]}$-least    enumeration of $\omega_1^{<\omega_1}$. For any $\beta<\omega_2$, define
$\delta_{\beta}^{\ast}=\{\alpha\in\omega_1\mid \exists\eta\in\omega_1(s_\alpha=\sigma_{\beta}^{\ast}\cap \eta)\}$.
It is easy to check that $\langle\delta_{\beta}^{\ast}: \beta<\omega_2\rangle$ is an almost disjoint sequence. By almost disjoint forcing, force $A_0\subseteq\omega_1$ over $L[B_0]$ to code  $B$ such that $\alpha\in B\Leftrightarrow |A_0\cap \delta_{\alpha}^{\ast}|<\omega_1$. The forcing preserves all cardinals.

In the following, we need that $\omega_2^{L[A_0]}=\omega_2^{L[B_0]}$ which motivates Claim \ref{key claim by Woodin}.\footnote{Our original definition of $B$ corresponds to the case $\gamma^{\ast}=\omega_2$ which can not make that $\omega_2^{L[A_0]}=\omega_2^{L[B_0]}$ holds.}

\begin{claim}\label{key claim by Woodin}
$\omega_2^{L[A_0]}=\gamma^{\ast}$.\footnote{I would like to thank W.Hugh Woodin for pointing out the problem in our original definition of $B$ and providing this key claim.} 
\end{claim}
\begin{proof}
Let $\lambda=\omega_2^{L[A_0]}$.  It follows from the definition of $(\sigma_{\alpha}^{\ast}: \alpha < \omega_2)$ that (i) $B_0\cap \lambda \in L[A_0]$ and hence (ii) $\lambda = \omega_2^{L[B_0\cap\lambda]}$. By (i) and (ii), we have (iii) $B \cap \lambda \in L[A_0]$. By the definition of $B$, it follows that (iv) $S \in L[A_0]$ where $S$ is the theory of $(L_{\gamma^{\ast}}[B_0], B_0)$ with parameters from $\lambda$.  From the definition of $B$ and the fact that $A_0$ codes $B$, by (iv) it follows that $L_{\lambda}[B_0] \prec L_{\gamma^{\ast}}[B_0]$ and so by (c) in the definition of $(B_0,\gamma^{\ast}), \lambda = \gamma^{\ast}$.
\end{proof}
 
Now we work in $L[A_0]$. Let $E=K\cap \{\eta\mid L_{\eta}[A_0]\prec L_{\omega_2}[A_0]\}$.  Let \[D=\{\gamma>\omega_1 \mid  (L_{\gamma}[A_0, E], E\cap\gamma)\prec (L_{\omega_2}[A_0,E], E)\}.\]
Note that $D\subseteq E$. Define  $F: \mathcal{P}(\omega_1)\rightarrow \mathcal{P}(\omega_1)$ as follows: If $y\subseteq\omega_1$ codes $\gamma$, then $F(y)\subseteq \omega_1$ codes $(\beta, E\cap\beta)$ where $\beta$ is the least element of $D$ such that $\beta>\gamma$(since $D$ is a club in $\omega_2$, such $\beta$ exists); If $y$ does not code an ordinal, let $F(y)=\emptyset$.

By the similar construction of $\langle\delta_{\beta}^{\ast}: \beta<\omega_2\rangle$, we can define an almost disjoint sequence $\langle\delta_{\beta}\mid \beta<\omega_2\rangle$ on $\omega_1$. We first define a sequence $\langle\sigma_{\beta}\mid \beta<\omega_2\rangle$ such that for each $\beta, \sigma_{\beta}$ is the $<_{L[A_0,E]}$-least $\sigma \subset \omega_1$ such that $\sigma$ has cardinality $\omega_1$ and $\sigma$ is different from $\sigma_{\alpha}$ for any $\alpha<\beta$. Let $\langle t_\alpha\mid \alpha\in\omega_1\rangle\in L[A_0,E]$ be a $<_{L[A_0,E]}$-least   enumeration of $\omega_1^{<\omega_1}$. Then $\langle\delta_{\beta}: \beta<\omega_2\rangle$ is a sequence of almost disjoint subset of $\omega_1$ where  $\delta_{\beta}=\{\alpha\in\omega_1\mid \exists \eta\in\omega_1(t_\alpha=\sigma_{\beta}\cap \eta)\}$.

Let $\langle x_{\alpha}\mid \alpha<\omega_2\rangle$ be the enumeration of $\mathcal{P}(\omega_1)$ in  $L[A_0,E]$ in the order of construction. Define
\[Z_{F}=\{\alpha\cdot\omega_1 + \beta\mid   \alpha<\omega_2\wedge \beta\in F(x_{\alpha})\}.\]
By almost disjoint forcing, we get $A_1\subseteq \omega_1$ such that $\beta\in Z_{F}\Leftrightarrow |A_1\cap \delta_{\beta}|<\omega_1$. Let $A=(A_0,A_1)$. The forcing preserves all cardinals.

Now we work in $L[B_0,A]$. Let $\alpha_{A}$ be the least $\alpha$ such that $L_{\alpha}[A]\models Z_3$. Note that $\omega_1^{L[A]}<\alpha_{A}<\omega_2^{L[A]}$ since $Z_3$ proves that $\omega_1$ exists.\footnote{Note that $\omega_1^{L[A]}=\omega_1^{L[B_0,A]}$ and $\omega_2^{L[A]}=\omega_2^{L[B_0,A]}$ by Claim \ref{key claim by Woodin}.} We show that in $L[B_0,A]$, \begin{equation}\label{key resut in section two}
\text{if $\omega_1\leq\alpha<\alpha_{A}$ is $A$-admissible, then $\alpha$ is an $L$-cardinal with the strong reflecting property.}
\end{equation}

By (\ref{any elem of K has srp}) and Proposition \ref{reflecting},$L_{\omega_2}[G,H]\models\omega_1$ has the strong reflecting property. By $(d)$ in the definition of $(B_0,\gamma^{\ast})$ and Proposition \ref{upward abso for reflecting proper}, $L[B_0,A]\models\omega_1$ has the strong reflecting property. Suppose $\omega_1<\alpha<\alpha_{A}$ is $A$-admissible. Define
\begin{equation}\label{def of gamma zero}
\gamma_0=\sup(\alpha\cap D).
\end{equation}

If $\alpha\cap D=\emptyset$, let $\gamma_0=0$.  Note that if $\gamma_0>0$, then $\gamma_0\in D$. We assume that $\gamma_0<\alpha$ and try to get a contradiction. It suffices to consider the case $\gamma_0>0$. Let $\alpha_0$ be the least $A_0$-admissible ordinal such that $\alpha_0>\gamma_0$. Since $\alpha$ is $A_0$-admissible, $\alpha_0\leq\alpha$.

\begin{claim}\label{cliam on eauaiont}
$E\cap\alpha_0=E\cap (\gamma_0+1).$
\end{claim}
\begin{proof}
We show that $E\cap\alpha_0\subseteq E\cap (\gamma_0+1)$. Suppose $\gamma\in E\cap\alpha_0$ and $\gamma>\gamma_0$. Since $\gamma\in E, L_{\gamma}[A_0]\prec L_{\omega_2}[A_0]$. Since $\alpha_0$ is definable from $\gamma_0$ and $A_0$, $\alpha_0$ is definable in $L_{\gamma}[A_0]$. So  $\alpha_0\leq\gamma$. Contradiction.
\end{proof}

By Claim \ref{cliam on eauaiont}, $L_{\alpha_0}[A_0, E]=L_{\alpha_0}[A_0, E\cap\gamma_0]$.

We need the following lemma to get that $L_{\gamma_0}[A_0, E\cap\gamma_0][A_1]=L_{\gamma_0}[A]$ in Claim \ref{key calim on countable}. 

\begin{lemma}\label{key claim in step}
$E\cap\gamma_0\in L_{\gamma_0+1}[A].$
\end{lemma}
\begin{proof}
We prove by induction that for any $\gamma\in D\cap\alpha_A, E\cap\gamma\in L_{\gamma+1}[A]$. Fix $\gamma\in D\cap\alpha_A$. Suppose for any $\gamma^{\prime}\in D\cap\gamma$, $E\cap\gamma^{\prime}\in L_{\gamma^{\prime}+1}[A]$. We show that $E\cap\gamma\in L_{\gamma+1}[A]$. If $\gamma\leq\omega_1$, this is trivial. Suppose $\gamma>\omega_1$. 

Case 1: There is $\gamma^{\prime}\in D$ such that $\gamma$ is the least element of $D$ such that $\gamma>\gamma^{\prime}$. Let $\eta$ be the least $A_0$-admissible ordinal such that $\eta>\gamma^{\prime}$. By the similar argument as Claim \ref{cliam on eauaiont}, $E\cap\eta=E\cap(\gamma^{\prime}+1)$. From our definitions, for any $\beta<\eta$ we have: (1) $\langle x_{\xi}\mid\xi\in\beta\rangle\in L_{\eta}[A_0, E]=L_{\eta}[A_0, E\cap\gamma^{\prime}]$; (2) $\langle \delta_{\xi}\mid\xi\in\beta\rangle\in L_{\eta}[A_0, E]=L_{\eta}[A_0, E\cap\gamma^{\prime}]$; (3) $\langle x_{\xi}\mid\xi\in\eta\rangle$ enumerates $\mathcal{P}(\omega_1)\cap L_{\eta}[A_0, E]=\mathcal{P}(\omega_1)\cap L_{\eta}[A_0, E\cap\gamma^{\prime}]$.

Suppose $y\subseteq\omega_1$ and $y\in L_{\eta}[A_0, E\cap\gamma^{\prime}]$. Then $y=x_{\xi}$ for some $\xi<\eta$. Note that $\xi\cdot\omega+\alpha<\eta$ for any $\alpha<\omega_1$.  $\alpha\in F(y)$ if and only if $|A_1\cap\delta_{\xi\cdot\omega+\alpha}|<\omega_1$. So $F(y)\in L_{\eta}[A_0, E\cap\gamma^{\prime}][A_1]$. Hence we have shown that if $y\in\mathcal{P}(\omega_1)\cap L_{\eta}[A_0, E\cap\gamma^{\prime}]$, then $F(y)\in L_{\eta}[A, E\cap\gamma^{\prime}]$.

\begin{claim}\label{first key calim on countable}
$L_{\eta}[A_0,E\cap\gamma^{\prime}]\models\gamma^{\prime}<\omega_2$.
\end{claim}
\begin{proof}
Suppose not. Then we have
\begin{equation}\label{key equation in sec two}
\text{$\gamma^{\prime}=\omega_2^{L_{\eta}[A_0, E\cap\gamma^{\prime}]}$.}
\end{equation}
Let $P$ be the partial order which codes $Z_{F}$ via $\langle\delta_{\beta}\mid \beta<\omega_2\rangle$.\footnote{$P=[\omega_1]^{<\omega_1}\times [Z_{F}]^{<\omega_1}$. $(p,q)\leq (p^{\prime},q^{\prime})$ \text{if{f}} $p\supseteq  p^{\prime}, q\supseteq q^{\prime}$ and $\forall\alpha\in q^{\prime}(p\cap \delta_{\alpha}\subseteq p^{\prime})$.} From our definitions of $E, F$ and $\langle x_{\alpha}\mid \alpha<\omega_2\rangle$, $P$ is a definable subset of $L_{\omega_2}[A_0, E]$. Standard argument gives that $P$ is $\omega_2$-c.c. in $L_{\omega_2}[A_0,E]$.\footnote{i.e. If $D\subseteq P$ is a maximal antichain with $D\in L_{\omega_2}[A_0,E]$, then $L_{\omega_2}[A_0,E]\models |D|\leq\omega_1$.}  Let $P^{\ast}=P\cap L_{\gamma^{\prime}}[A_0,E]$. Since $\gamma^{\prime}\in D$, 
\begin{equation}\label{property of gamma zero}
(L_{\gamma^{\prime}}[A_0,E], E\cap\gamma^{\prime})\prec (L_{\omega_2}[A_0,E], E).
\end{equation}
Suppose $D^{\ast}\subseteq P^{\ast}$ is a maximal antichain with $D^{\ast}\in L_{\gamma^{\prime}}[A_0,E]$. Then by (\ref{property of gamma zero}), $D^{\ast}$ is a maximal antichain in $P$. Since $L_{\omega_2}[A_0,E]\models |D^{\ast}|\leq\omega_1$, by (\ref{property of gamma zero}), $L_{\gamma^{\prime}}[A_0,E]\models |D^{\ast}|\leq\omega_1$. So $P^{\ast}$ is $\omega_2$-c.c. in $L_{\gamma^{\prime}}[A_0,E]$. 
By (\ref{key equation in sec two}), 
\begin{equation}\label{equation on two}
\text{$L_{\eta}[A_0, E\cap\gamma^{\prime}]\cap 2^{\omega_1}=L_{\gamma^{\prime}}[A_0, E\cap\gamma^{\prime}]\cap 2^{\omega_1}$.}
\end{equation}
Since $P^{\ast}$ is $\omega_2$-c.c. in $L_{\gamma^{\prime}}[A_0,E]$, by (\ref{equation on two}), $P^{\ast}$ is $\omega_2$-c.c in $L_{\eta}[A_0, E\cap\gamma^{\prime}]$.

We show that $A_1$ is generic over $L_{\eta}[A_0, E\cap\gamma^{\prime}]$ for $P^{\ast}$. Let $Y\subseteq P^{\ast}$ be a maximal antichain with $Y\in L_{\eta}[A_0, E\cap\gamma^{\prime}]$. Since $P^{\ast}$ is $\omega_2$-c.c in $L_{\eta}[A_0, E\cap\gamma^{\prime}]$, by (\ref{key equation in sec two}),    $Y\in L_{\gamma^{\prime}}[A_0, E\cap\gamma^{\prime}]$. By (\ref{property of gamma zero}), $Y$ is a maximal antichain in $P$. So the filter given by $A_1$ meets $Y$.

Note that $\gamma^{\prime}=\omega_2^{L_{\eta}[A_0, E\cap\gamma^{\prime}]}=\omega_2^{L_{\eta}[A_0, E\cap\gamma^{\prime}][A_1]}$. Since $\gamma^{\prime}\in D$, by induction hypothesis $L_{\gamma^{\prime}}[A_0, E\cap\gamma^{\prime}][A_1]=L_{\gamma^{\prime}}[A]$. So  $L_{\gamma^{\prime}}[A]\models Z_3$ which contradicts the minimality of $\alpha_{A}$.
\end{proof}

Take $y\in L_{\eta}[A_0, E\cap\gamma^{\prime}]\cap\mathcal{P}(\omega_1)$ such that $y$ codes $\gamma^{\prime}$. So $F(y)$ codes $(\gamma,  E\cap\gamma)$ and $F(y)\in L_{\eta}[A, E\cap\gamma^{\prime}]$. Then $F(y)$ is definable in $L_{\gamma}[A, E\cap\gamma^{\prime}]$. By induction hypothesis, $F(y)\in L_{\gamma+1}[A]$. Since $F(y)$ codes $E\cap\gamma$, $E\cap\gamma\in L_{\gamma+1}[A]$.

Case 2: $\gamma$ is the least element of $D$. Take $y\in L_{\omega_1}[A_0, E]\cap\mathcal{P}(\omega_1)$ such that $y$ codes $0$. Then $y=x_0$. Since $\gamma$ is the least element of $D$ such that $\gamma>0$, $F(y)$ codes $E\cap\gamma$. Note that for any $\beta<\omega_1,\langle\delta_{\xi}\mid\xi\in\beta\rangle\in L_{\omega_1}[A_0, E]$ and $\alpha\in F(y)$ if and only if $|A_1\cap\delta_{\alpha}|$ is countable. So 
$F(y)$ is definable in $L_{\omega_1}[A, E]$. Since $E\cap \omega_1=\emptyset$, 
$F(y)\in L_{\gamma+1}[A]$. Since $F(y)$ codes $E\cap\gamma$, $E\cap\gamma\in L_{\gamma+1}[A]$.

Case 3: $\gamma$ is a limit point of $D$. Then standard argument gives that $E\cap\gamma\in L_{\gamma+1}[A]$ by induction hypothesis. 

Since $\gamma_0\in D\cap\alpha_A$, we have $E\cap\gamma_0\in L_{\gamma_0+1}[A]$.
\end{proof}

\begin{claim}\label{key calim on countable}
$L_{\alpha_0}[A_0, E\cap\gamma_0]\models \gamma_0<\omega_2$.
\end{claim}
\begin{proof}
The proof is essentially the same as Claim \ref{first key calim on countable}(replace $\eta$ by $\alpha_0$ and $\gamma^{\prime}$ by $\gamma_0$). Suppose not. Then $\gamma_0=\omega_2^{L_{\alpha_0}[A_0, E\cap\gamma_0]}$. Let $P$ be the partial order which codes $Z_{F}$ via $\langle\delta_{\beta}\mid \beta<\omega_2\rangle$ and $P^{\ast}=P\cap L_{\gamma_0}[A_0,E]$.
By the similar argument as Claim \ref{first key calim on countable}, we can show that $A_1$ is generic over $L_{\alpha_0}[A_0, E\cap\gamma_0]$ for $P^{\ast}$. Since $\gamma_0=\omega_2^{L_{\alpha_0}[A_0, E\cap\gamma_0]}=\omega_2^{L_{\alpha_0}[A_0, E\cap\gamma_0][A_1]}$ and by Lemma \ref{key claim in step} $L_{\gamma_0}[A_0, E\cap\gamma_0][A_1]=L_{\gamma_0}[A]$, we have $L_{\gamma_0}[A]\models Z_3$ which contradicts the minimality of $\alpha_{A}$.
\end{proof}

From our definitions, we have
\begin{equation}\label{def of almost disjonit systen}
\text{for $\eta<\alpha_0, \langle\delta_{\beta}: \beta<\eta\rangle\in L_{\alpha_0}[A_0, E]=L_{\alpha_0}[A_0, E\cap\gamma_0]$ and}
\end{equation}
\begin{equation}\label{eumer thid}
\text{$\langle x_{\beta}\mid \beta<\alpha_0\rangle$ enumerates $\mathcal{P}(\omega_1)\cap L_{\alpha_0}[A_0, E]=\mathcal{P}(\omega_1)\cap L_{\alpha_0}[A_0, E\cap\gamma_0]$.}
\end{equation}

\begin{claim}\label{ley claim on definalbity}
If $y\subseteq\omega_1$ and $y\in L_{\alpha_0}[A_0, E\cap\gamma_0]$, then $F(y)\in L_{\alpha_0}[A]$.
\end{claim}
\begin{proof}
Suppose $y\in \mathcal{P}(\omega_1)\cap L_{\alpha_0}[A_0, E\cap\gamma_0]$. By (\ref{eumer thid}), $y=x_{\xi}$ for some $\xi<\alpha_0$. Note that $\xi\cdot\omega_1 +\alpha<\alpha_0$ for  $\alpha<\omega_1$. Then $\alpha\in F(y)$ \text{if{f}} $\xi\cdot\omega_1 +\alpha\in Z_{F}$ \text{if{f}} $|A_1\cap\delta_{\xi\cdot\omega_1+\alpha}|<\omega_1$. By  (\ref{def of almost disjonit systen}), $F(y)\in L_{\alpha_0}[A_0, E\cap\gamma_0][A_1]$. Since by  Lemma \ref{key claim in step}, $E\cap\gamma_0\in L_{\gamma_0+1}[A]$,  $L_{\alpha_0}[A_0, E\cap\gamma_0][A_1]=L_{\alpha_0}[A]$. Hence $F(y)\in L_{\alpha_0}[A]$.
\end{proof}

By Claim \ref{key calim on countable}, there exists  $y\in L_{\alpha_0}[A_0, E\cap\gamma_0]\cap\mathcal{P}(\omega_1)$ such that $y$ codes $\gamma_0$. By the definition of $F$, $F(y)$ codes $\gamma_1$ where $\gamma_1$ is the least element of $E$ such that $\gamma_1>\gamma_0$ and
\begin{equation}\label{least elem of such that}
(L_{\gamma_1}[A_0, E], E\cap\gamma_1)\prec (L_{\omega_2}[A_0,E], E).
\end{equation}

By Claim \ref{ley claim on definalbity}, $F(y)\in L_{\alpha_0}[A]$. Since $F(y)$ codes $\gamma_1$, $\gamma_1<\alpha_0$. Since $\alpha_0\leq\alpha, \gamma_1<\alpha$. By (\ref{least elem of such that}) and (\ref{def of gamma zero}), $\gamma_1\leq\gamma_0$. Contradiction.

So the assumption that $\gamma_0<\alpha$ is false. Then $\gamma_0=\alpha$ and hence $\alpha\in E$. By (\ref{any elem of K has srp}) and Proposition \ref{reflecting}, $L_{\omega_2}[G,H]\models\alpha$ has the strong reflecting property. By $(d)$ in the definition of $(B_0,\gamma^{\ast})$ and Proposition \ref{upward abso for reflecting proper}, $L[B_0,A]\models\alpha$ has the strong reflecting property. We have proved $L[B_0,A]\models$ (\ref{key resut in section two}).

We still work in $L[B_0,A]$. Suppose $Y\prec L_{\alpha_{A}}[A], |Y|=\omega$ and $\bar{Y}$ is the transitive collapse of $Y$. Let $\bar{\omega_1}=Y\cap\omega_1$. Then $\bar{Y}=L_{\bar{\alpha}}[\bar{A}]$ where $\bar{A}=A\cap \bar{\omega_1}$ and  $\bar{\alpha}=o.t.(Y\cap \alpha_{A})$. Note that $\bar{\omega_1}<\omega_1$ and $L_{\bar{\alpha}}[\bar{A}]\models Z_3$. Suppose $\bar{\omega_1}\leq\eta<\bar{\alpha}$ is $\bar{A}$-admissible. By (\ref{key resut in section two}), $\eta$ is an $L$-cardinal.  Let
\begin{center}
$Z=\{\delta<\omega_1\mid \exists \alpha>\delta(L_{\alpha}[A\cap \delta]\models ``Z_3\, + \,\delta=\omega_1" \wedge \forall \eta ((\delta\leq\eta<\alpha \wedge \eta$ is $A\cap \delta$-admissible)$\rightarrow \eta$ is an $L$-cardinal))\}.
\end{center}

Let $Q=\{Y\cap\omega_1\mid Y\prec L_{\alpha_{A}}[A]\wedge  |Y|=\omega\}$. We have shown that $Q\subseteq Z$ and hence $Z$ contains a club in $\omega_1$.  Define
\begin{equation}\label{def of S}
\text{$S=Z\cap \{\alpha<\omega_1: \alpha$ is an $L$-cardinal\}.}
\end{equation}
Then $S$ is stationary and in fact contains a club.

\subsection{Step Three}\label{big thm step four}

In this step, we shoot a club $C$ through $S$ via Baumgartner's forcing $P_{S}^{B}$ such that if $\eta$ is the limit point of $C$ and $L_{\alpha_{\eta}}[A\cap\eta, C\cap\eta]\models\eta=\omega_1$, then $L_{\alpha_{\eta}}[A\cap\eta, C\cap\eta]\models Z_3$ where $\alpha_{\eta}$ is  the least $\alpha>\eta$ such that $L_{\alpha}[A\cap\eta]\models Z_3\, + \, \eta=\omega_1$.\footnote{We failed to shoot such a club via variants of Harrington's forcing. The key point is that Theorem \ref{second pre thm} works for  $P_{S}^{B}$ but does not work for $P_{S}$.}

We still work in $L[B_0,A]$. For $f\in P_{S}^{B}$, define $(P_{S}^{B})_{f}=\{g\in P_{S}^{B}\mid g\leq f$ and  $\max(dom(g))=\max(dom(f))\}$.  For  $\eta<\omega_1$, define $P_{S}^{B}\upharpoonright\eta=\{f\in P_{S}^{B}\mid (dom(f)\cup ran(f))\subseteq \eta\}$.

\begin{lemma}\label{Bamber forcing key lemma}
Suppose $f \in P_{S}^{B}$. Then $f\Vdash_{P_{S}^{B}} \dot{G}\cap (P_{S}^{B})_{f}$ is $(P_{S}^{B})_{f}$-generic over $V$.
\end{lemma}
\begin{proof}
Suppose $h\in P_{S}^{B}, h\leq f$ and $D$ is a dense subset of $(P_{S}^{B})_{f}$. It suffices to show that there is $p \in D$ such that $h\cup p\in P_{S}^{B}$. Let $\max(dom(f))=\beta$. Then $h\upharpoonright (\beta+1)\in (P_{S}^{B})_{f}$. Take $p\in D$  such that  $p\leq h\upharpoonright (\beta+1)$. We show that $h\cup p\in P_{S}^{B}$.

Let $\alpha=\max(dom(h))$. Since $h\in P_{S}^{B}$, there exists $E \subseteq S$ such that $E$ is closed, $o.t.(E)=\alpha+1$  and for any $\gamma\in dom(h), h(\gamma)$ is the $\gamma$-th element of $E$. Since $p\in (P_{S}^{B})_{f}, \max(dom(p))=\beta$. Let $F\subseteq S$ be closed such that  $o.t.(F)=\beta+1$ and for any $\gamma\in dom(p), p(\gamma)$ is the $\gamma$-th element of $F$.
Note that $h(\beta)=f(\beta)=p(\beta)$. Let $C=\{\gamma\in E\mid \gamma\geq h(\beta)=p(\beta)\}\cup F$. $C\subseteq S$ is closed. Since $p\leq h\upharpoonright (\beta+1), o.t.(C)=\alpha+1$.  For any $\gamma\in dom(h\cup p), (h\cup p)(\gamma)$ is the $\gamma$-th element of $C$. So $h\cup p\in P_{S}^{B}$.
\end{proof}

\begin{lemma}\label{lemma for tree representation}
Suppose $f\in P_{S}^{B}$ where $f=\{(\eta, \eta)\}$. Then
\[(P_{S}^{B})_f=\{g\cup \{(\eta, \eta)\} \mid g\in P_{S}^{B}\upharpoonright\eta\}.\]
\end{lemma}
\begin{proof}
$\subseteq$ is trivial. Fix $g\in P_{S}^{B}\upharpoonright\eta$. We show that $g\cup \{(\eta, \eta)\}\in P_{S}^{B}$. It suffices to show that there exists $H: \eta+1\rightarrow S\cap (\eta+1)$ such that
\begin{equation}\label{condtion of function}
\text{$H$ is increasing and continuous, $H$ extends $g$ and $H(\eta)=\eta$.}
\end{equation}
Let $\xi=\max(dom(g))$. Let $F: \xi+1\rightarrow S\cap (g(\xi)+1)$ be the witness function for $g\in P_{S}^{B}$(i.e. $F$ is increasing, continuous and extends $g$). Let $E: \eta+1\rightarrow S\cap (\eta+1)$ be the witness function for $f\in P_{S}^{B}$(i.e. $E$ is increasing, continuous and $E(\eta)=\eta$). Let $C=ran(E)\setminus (g(\xi)+1)$. Since $\eta\in S, \eta$ is indecomposable\footnote{A limit ordinal $\gamma$ is indecomposable if there is no $\alpha<\gamma$ and $\beta<\gamma$ such that $\alpha+\beta=\gamma$. Note that if $\gamma$ is indecomposable, then for any $\alpha<\gamma, o.t.(\{\beta\mid \alpha\leq\beta<\gamma\})=\gamma$.} and hence $o.t.(C)=o.t.((\eta+1)\setminus (g(\xi)+1))=\eta+1$ since $g(\xi)<\eta$. Let $\pi: \eta+1\rightarrow C$ be an increasing continuous enumeration of $C$. Define $H: \eta+1\rightarrow S\cap (\eta+1)$ by $H\upharpoonright \xi+1=F$ and for any $\alpha\leq\eta, H(\xi+1+\alpha)=\pi(\alpha)$. It is easy to check that $H$ satisfies (\ref{condtion of function}).
\end{proof}

\begin{notation}
For  $\eta\in S$, let $\alpha_{\eta}$ be the least $\alpha>\eta$ such that $L_{\alpha}[A\cap\eta]\models Z_3\, + \, \eta=\omega_1$.
\end{notation}

\begin{lemma}\label{Basic properties of S}
\begin{enumerate}[(a)]
      \item Suppose $\eta\in S$ and $\beta<\eta$. Then $L_{\eta}[A]\models \beta$ is countable.
  \item Suppose $\eta_0, \eta_1\in S$ and $\eta_0<\eta_1$. Then $\alpha_{\eta_0}<\eta_1$. i.e. For any $\eta\in S, \alpha_{\eta}<\bar{\eta}$ where $\bar{\eta}=\min(S\setminus(\eta+1))$.
\end{enumerate}
\end{lemma}
\begin{proof}
(a) Since $\eta \in S, L_{\alpha_{\eta}}[A\cap\eta]\models \eta=\omega_1$. Note that $\mathbb{R}\cap L_{\alpha_{\eta}}[A\cap\eta]=\mathbb{R}\cap L_{\eta}[A\cap\eta]=\mathbb{R}\cap L_{\eta}[A]$. Since $\beta<\eta, L_{\eta}[A]=L_{\eta}[A\cap\eta]\models \beta$ is countable.

(b) Suppose $\eta_1\leq\alpha_{\eta_0}$. Note that $Z_3\vdash \forall E\subseteq \omega_1(L_{\omega_1}[E]\models ZFC^{-})$. Since $L_{\alpha_{\eta_1}}[A\cap\eta_1]\models Z_3\, + \,\eta_1= \omega_1, L_{\eta_1}[A\cap\eta_0]\models ZFC^{-}$. Since $\eta_1\leq\alpha_{\eta_0}$ and $L_{\alpha_{\eta_0}}[A\cap\eta_0]\models\eta_0= \omega_1, L_{\eta_1}[A\cap\eta_0]\subseteq L_{\alpha_{\eta_0}}[A\cap\eta_0]$ and hence  $L_{\eta_1}[A\cap\eta_0]\models  \eta_0=\omega_1$. Since $\eta_1\in S$, $L_{\eta_1}[A\cap\eta_0]\models ZFC^{-}$, $L_{\eta_1}[A\cap\eta_0]\subseteq L_{\alpha_{\eta_0}}[A\cap\eta_0]\models Z_3$ and $L_{\eta_1}[A\cap\eta_0]\models  \eta_0=\omega_1$, we have $L_{\eta_1}[A\cap\eta_0]\models Z_3$. i.e.
\begin{equation}\label{and hence eta}
L_{\eta_1}[A\cap\eta_0]\models Z_3\, + \,\eta_0=\omega_1.
\end{equation}
So $\eta_1\geq \alpha_{\eta_0}$ and hence $\eta_1=\alpha_{\eta_0}$.

\begin{fact}\label{apply above fact}
(Folklore)\quad ($Z_3$) \quad $\forall E\subseteq \omega_1\, \forall\alpha<\omega_1\, \forall a\in L_{\omega_1}[E]\, \exists X(X\prec L_{\omega_1}[E]\wedge |X|=\omega\wedge \alpha\cup\{a\}\subseteq X)$.\footnote{This fact is standard and its proof uses the standard Skolem Hull argument. We only need to check that the proof can be run in $Z_3$. It is not hard to check this.}
\end{fact}
Since $L_{\alpha_{\eta_1}}[A\cap\eta_1]\models Z_3\, + \,\eta_1=\omega_1$, by Fact \ref{apply above fact}, there is  $X\in L_{\alpha_{\eta_1}}[A\cap\eta_1]$ such that $X\prec L_{\eta_1}[A\cap\eta_0],  L_{\alpha_{\eta_1}}[A\cap\eta_1]\models |X|=\omega, A\cap\eta_0\in X$ and $\eta_0+1\subseteq X$(in Fact \ref{apply above fact} let $E=A\cap\eta_0, \alpha=\eta_0+1$ and $a=A\cap\eta_0$). Let $M$ be the transitive collapse of $X$ and $M=L_{\bar{\eta_1}}[A\cap\eta_0]$. Note that $\eta_0<\bar{\eta_1}<\eta_1$. By (\ref{and hence eta}), $L_{\bar{\eta_1}}[A\cap\eta_0]\models Z_3\, + \,\eta_0=\omega_1$ and hence $\alpha_{\eta_0}\leq \bar{\eta_1}<\eta_1$. Contradiction.
\end{proof}

\begin{theorem}\label{second pre thm}
Suppose $\{(\eta,\eta)\}\in P_{S}^{B}$. Then
$(P_{S\cap\eta}^{B})^{L_{\alpha_{\eta}}[A\cap \eta]}=P_{S}^{B}\upharpoonright\eta$.
\end{theorem}
\begin{proof}
$\subseteq$ is trivial. Suppose $g\in P_{S}^{B}\upharpoonright\eta$. We show that $g\in   (P_{S\cap\eta}^{B})^{L_{\alpha_{\eta}}[A\cap \eta]}$. Let $\xi=\max(dom(g))$. Let $H: \xi+1\rightarrow S\cap\eta$ be the witness function for $g\in P_{S}^{B}$(i.e. $H$ is increasing, continuous and extends $g$). It suffices to find a function $H^{\infty}: \xi+1\rightarrow S\cap\eta$ such that
\begin{equation}\label{property of fun we want}
\text{$H^{\infty}$ is increasing, continuous, $H^{\infty}$ extends $g$ and $H^{\infty}\in L_{\alpha_{\eta}}[A\cap \eta]$.}
\end{equation}

Pick a surjection $e_0: \omega\rightarrow \xi+1$ such that $e_0\in L_{\alpha_{\eta}}[A\cap \eta]$ and
\begin{equation}\label{infinite condiition}
\text{for any $\alpha\leq\xi, \{i\in\omega\mid  e_0(i)=\alpha\}$ is infinite.}
\end{equation}
Pick a surjection $e_1: \omega\rightarrow H(\xi)+1$ such that $e_1\in L_{\alpha_{\eta}}[A\cap \eta]$. Let $T$ be the set of all pairs $(\pi_1, \pi_2)$ such that  $\pi_1: k\rightarrow (H(\xi)+1)\cap S$ where $k\in\omega$, $\pi_2: k\rightarrow\omega$ and the following hold:\footnote{The tree $T$ is defined for definability argument. We define $T$ to show that $H^{\infty}\in L_{\alpha_{\eta}}[A\cap \eta]$: we first show that $T\in L_{\alpha_{\eta}}[A\cap \eta]$ and then show that $H^{\infty}\in L_{\alpha_{\eta}}[A\cap \eta]$ via Claim \ref{there is a such tree}.}

\begin{equation}\label{first con on tree}
\text{For all $i<k$, if $e_0(i)\in dom(g)$, then $\pi_1(i)=g(e_0(i))$; }
\end{equation}
\begin{equation}\label{second con on tree}
\text{$\forall\, i<j<k(\pi_1(i)=\pi_1(j)\Leftrightarrow e_0(i)=e_0(j))$;}
\end{equation}
\begin{equation}\label{increasing con}
\text{$\forall\, i<j<k(\pi_1(i)<\pi_1(j)\Leftrightarrow e_0(i)<e_0(j))$;}
\end{equation}
For all $i<k$, if $e_0(i)>0$ is a limit ordinal and $\pi_2(i)<k$, then
\begin{equation}\label{fourth con on tree}
\text{$\sup(\{e_1(m)\mid m\leq i\wedge  e_1(m)<\pi_1(i)\})<\pi_1(\pi_2(i))<\pi_1(i)$ and $e_0(\pi_2(i))<e_0(i)$.}
\end{equation}

By (\ref{def of S}) and  Lemma \ref{Basic properties of S}(b), $S\cap (H(\xi)+1)\in L_{\alpha_{\eta}}[A\cap \eta]$.  Since $g\in P_{S}^{B}\upharpoonright\eta, g\in L_{\alpha_{\eta}}[A\cap \eta]$. Since $S\cap (H(\xi)+1), g, e_0, e_1\in L_{\alpha_{\eta}}[A\cap \eta]$, by the definition of $T$, $T\in L_{\alpha_{\eta}}[A\cap \eta]$.

Define $\pi^{\infty}_1:\omega\rightarrow (H(\xi)+1)\cap S$  as follows: $\pi^{\infty}_1(i)=H(e_0(i))$ for  $i\in\omega$. Now we define $\pi^{\infty}_2: \omega\rightarrow\omega$ as follows such that for all $i<\omega$, if $e_0(i)>0$ is a limit ordinal, then
\begin{equation}\label{key if limit}
\text{$\sup(\{e_1(m)\mid m\leq i\wedge  e_1(m)<\pi^{\infty}_1(i)  \})<\pi^{\infty}_1(\pi^{\infty}_2(i))<\pi^{\infty}_1(i)$
          and $e_0(\pi^{\infty}_2(i))<e_0(i)$.}
\end{equation}

Suppose $e_0(i)>0$ and $e_0(i)$ is a limit ordinal. Let $\alpha=e_0(i)$. Since $H$ is continuous, $H(\alpha)$ is a limit ordinal. Let $\beta<\alpha$ be the least ordinal such that $\sup(\{e_1(m)\mid m\leq i\wedge  e_1(m)<\pi^{\infty}_1(i)\})<H(\beta)<H(\alpha)$. Let $\pi^{\infty}_2(i)$ be the least $j\in\omega$ such that $e_0(j)=\beta$. If $e_0(i)=0$ or $e_0(i)$ is not a limit ordinal, let $\pi^{\infty}_2(i)=0$. Since $\pi^{\infty}_1(\pi^{\infty}_2(i))=\pi^{\infty}_1(j)=H(e_0(j))=H(\beta), \pi^{\infty}_1(i)=H(\alpha)$ and $e_0(\pi^{\infty}_2(i))=\beta<\alpha=e_0(i)$, (\ref{key if limit}) holds.\footnote{To show (\ref{key if limit}), we use that $H$ is continuous.}

\begin{claim}\label{there is a such tree}
For any $k\in\omega, (\pi^{\infty}_1\upharpoonright k, \pi^{\infty}_2\upharpoonright k)\in T$.
\end{claim}
\begin{proof}
Fix $k\in\omega$. We show that $(\pi^{\infty}_1\upharpoonright k, \pi^{\infty}_2\upharpoonright k)$ satisfies conditions (\ref{first con on tree})-(\ref{fourth con on tree})  in the definition of $T$. Since $H$ extends $g$,  (\ref{first con on tree}) holds. Since $H$ is strictly increasing, (\ref{second con on tree}) and (\ref{increasing con}) hold.  By (\ref{key if limit}), (\ref{fourth con on tree}) holds.
\end{proof}

Define $H^{\infty}: \xi+1\rightarrow S\cap\eta$ by
\begin{equation}\label{def of fun we want}
\text{$H^{\infty}(e_0(i))=\pi^{\infty}_1(i)$ for  $i\in\omega$.}
\end{equation}
We show that $H^{\infty}$ satisfies (\ref{property of fun we want}). By (\ref{second con on tree}), $H^{\infty}$ is well defined. By (\ref{increasing con}), $H^{\infty}$ is increasing. By (\ref{first con on tree}), $H^{\infty}$ extends $g$. Since $T, e_0\in L_{\alpha_{\eta}}[A\cap \eta]$,  by (\ref{def of fun we want}) and Claim \ref{there is a such tree},  $H^{\infty}\in L_{\alpha_{\eta}}[A\cap \eta]$.

\begin{claim}
$H^{\infty}$ is continuous.
\end{claim}
\begin{proof}
Suppose $0<\alpha\leq\xi$ is a limit ordinal. We show that $H^{\infty}(\alpha)=\sup(\{H^{\infty}(\beta)\mid\beta<\alpha\})$. Suppose not. Then there exists $\theta$ such that $\sup(\{H^{\infty}(\beta)\mid\beta<\alpha\})<\theta<H^{\infty}(\alpha)$.

Pick $m_0$ such that $e_1(m_0)=\theta$. By (\ref{infinite condiition}), pick  $i>m_0$ such that $e_0(i)=\alpha$. Since $e_1(m_0)=\theta<H^{\infty}(\alpha)$, by (\ref{def of fun we want}), $\theta\leq\sup(\{e_1(m)\mid m\leq i\wedge  e_1(m)<\pi^{\infty}_1(i)\})$. By (\ref{def of fun we want}), $\pi^{\infty}_1(\pi^{\infty}_2(i))=H^{\infty}(e_0(\pi^{\infty}_2(i)))$. By (\ref{key if limit}), $\theta<H^{\infty}(e_0(\pi^{\infty}_2(i)))$ and $e_0(\pi^{\infty}_2(i))<e_0(i)=\alpha$. But $\sup(\{H^{\infty}(\beta)\mid\beta<\alpha\})<\theta$. Contradiction.\footnote{The proof of Theorem \ref{second pre thm} depends on (\ref{def of S}) and property of Baumgartner's forcing. In fact, its proof only uses the part  $(\forall\eta\in S)(\exists\delta>\eta(L_{\delta}[A\cap\eta]\models  Z_3\, + \,\eta=\omega_1))$ in (\ref{def of S}).}
\end{proof}
\end{proof}

\begin{theorem}\label{B forcing key representation theorem}
Suppose $f\in P_{S}^{B}$ where $f=\{(\eta, \eta)\}$. Then
\[(P_{S}^{B})_{f}=\{ g\cup \{(\eta, \eta)\}\mid g\in   (P_{S\cap\eta}^{B})^{L_{\alpha_{\eta}}[A\cap \eta]}\}. \]
\end{theorem}
\begin{proof}
Follows from Lemma \ref{lemma for tree representation} and Theorem \ref{second pre thm}.
\end{proof}

Suppose $G^{\ast}$ is $P_{S}^{B}$-generic over $L[B_0,A]$. Define $F_{G^{\ast}}=\bigcup\{f\mid f\in G^{\ast}\}$. Then $F_{G^{\ast}}: \omega_1\rightarrow S$ is increasing and continuous. Let $C=ran(F_{G^{\ast}})$. Then $C\subseteq S$ is a club in $\omega_1$. Let $Lim(C)=\{\alpha\mid\alpha$ is a limit point of $C$\}. Now we work in $L[B_0,A,C]$.

\begin{fact}\label{forcing preserving system lemma}
(Folklore, \cite{Woodinfirst})\quad Suppose $M \models  Z_3, P\in M$ is a forcing notion, $M\models |P|\leq\omega_1$ and $G$ is $P$-generic over $M$. If $M\models P$ preserves $\omega_1$, then $M[G]\models Z_3$.
\end{fact}

\begin{theorem}\label{thoerem on the key property of B forcing}
Suppose $\eta\in Lim(C)$. Then
\[L_{\alpha_{\eta}}[A\cap\eta, C\cap\eta]\models\eta=\omega_1\Leftrightarrow L_{\alpha_{\eta}}[A\cap\eta, C\cap\eta]\models Z_3.\]
\end{theorem}
\begin{proof}
$(\Rightarrow)$  Suppose $L_{\alpha_{\eta}}[A\cap\eta, C\cap\eta]\models\eta=\omega_1$. Then
\begin{equation}\label{fact on there is club}
\text{$L_{\alpha_{\eta}}[A\cap\eta, C\cap\eta]\models C\cap \eta$ is a club in $\eta$.}
\end{equation}
We show that
\begin{equation}\label{fact on stationary}
\text{$L_{\alpha_{\eta}}[A\cap\eta]\models S\cap\eta$ is stationary.}
\end{equation}

Suppose not. Then there exists a club $E$ in $\eta$ such that $E\in L_{\alpha_{\eta}}[A\cap \eta]$ and $E\cap S\cap\eta=\emptyset$. Then  $L_{\alpha_{\eta}}[A\cap\eta, C\cap\eta]\models E$ and $C\cap\eta$ are disjoint closed subsets of $\eta$. Contradiction.

By (\ref{fact on there is club}), $o.t.(C\cap \eta)=\eta$ and hence $\eta$ is the $\eta$-th element of $C$.  Since $F_{G^{\ast}}(\xi)$ is the $\xi$-th element of $C$, $F_{G^{\ast}}(\eta)=\eta$. Let $f=\{(\eta, \eta)\}$. Since $f\in G^{\ast}$, by Lemma \ref{Bamber forcing key lemma}, $G^{\ast}\cap (P_{S}^{B})_{f}$ is $(P_{S}^{B})_{f}$-generic over $V$. By Theorem \ref{B forcing key representation theorem},  $(P_{S}^{B})_{f}=\{h\cup \{(\eta, \eta)\}\mid h\in   (P_{S\cap\eta}^{B})^{L_{\alpha_{\eta}}[A\cap \eta]}\}$. So $G^{\ast}\cap (P_{S\cap\eta}^{B})^{L_{\alpha_{\eta}}[A\cap \eta]}$ is $(P_{S\cap\eta}^{B})^{L_{\alpha_{\eta}}[A\cap \eta]}$-generic over $L_{\alpha_{\eta}}[A\cap \eta]$ and hence
\begin{equation}\label{key generic fact}
\text{$C\cap\eta$ is $(P_{S\cap\eta}^{B})^{L_{\alpha_{\eta}}[A\cap \eta]}$-generic over $L_{\alpha_{\eta}}[A\cap \eta]$.}
\end{equation}

By (\ref{fact on stationary}), do  Baumgartner's forcing $P_{S\cap\eta}^{B}$ over $L_{\alpha_{\eta}}[A\cap \eta]$. Since $L_{\alpha_{\eta}}[A\cap\eta]\models Z_3$, by Fact \ref{Baumgartner fact on forcing notion},  $L_{\alpha_{\eta}}[A\cap\eta]\models ``|(P_{S\cap\eta}^{B})|=\omega_1$ and $P_{S\cap\eta}^{B}$ preserves $\omega_1"$. By (\ref{key generic fact}) and Fact \ref{forcing preserving system lemma}, $L_{\alpha_{\eta}}[A\cap\eta, C\cap\eta]\models Z_3$.

$(\Leftarrow)$ Suppose $L_{\alpha_{\eta}}[A\cap\eta, C\cap\eta]\models Z_3$. We show that $L_{\alpha_{\eta}}[A\cap\eta, C\cap\eta]\models\eta=\omega_1$. Suppose not. i.e. $\eta<\omega_1^{L_{\alpha_{\eta}}[A\cap\eta, C\cap\eta]}$. Since $L_{\alpha_{\eta}}[A\cap\eta]\subseteq L_{\alpha_{\eta}}[A\cap\eta, C\cap\eta]$, $\omega_1^{L_{\alpha_{\eta}}[A\cap\eta, C\cap\eta]}$ is a cardinal in $L_{\alpha_{\eta}}[A\cap\eta]$. But since $L_{\alpha_{\eta}}[A\cap\eta]\models Z_3\, + \,\eta=\omega_1, \eta=\omega_1^{L_{\alpha_{\eta}}[A\cap\eta]}$ is the largest cardinal in $L_{\alpha_{\eta}}[A\cap\eta]$. Contradiction.\footnote{The key step in  Theorem \ref{thoerem on the key property of B forcing} is to show that (\ref{fact on stationary}) implies (\ref{key generic fact}) which depends on the representation theorem for $(P_{S\cap\eta}^{B})^{L_{\alpha_{\eta}}[A\cap \eta]}$(Theorem \ref{second pre thm}).}
\end{proof}

As a summary, by (\ref{def of S}) and Theorem \ref{thoerem on the key property of B forcing}, $Lim(C)$ has the following properties:
\begin{equation}\label{first pro of c}
\text{$\forall\eta\in Lim(C)(\eta$ is an $L$-cardinal);}
\end{equation}
\begin{equation}\label{sec pro of c}
\text{$\forall\eta\in Lim(C)((\eta\leq\beta<\alpha_{\eta}\wedge\beta$ is $A\cap\eta$-admissible)$\rightarrow\beta$ is an $L$-cardinal);}
\end{equation}
\begin{equation}\label{thrd pro of c}
\text{$\forall\eta\in Lim(C)(L_{\alpha_{\eta}}[A\cap\eta, C\cap\eta]\models\eta=\omega_1\rightarrow L_{\alpha_{\eta}}[A\cap\eta, C\cap\eta]\models Z_3)$.}
\end{equation}

\subsection{Step Four}
In this step, we use properties of $Lim(C)$ to define the almost disjoint system on $\omega$ and some $B^{\ast}\subseteq \omega_1$. Then we do almost disjoint forcing to code $B^{\ast}$ by a real $x$. Finally, we use (\ref{first pro of c})-(\ref{thrd pro of c}) to show that $x$ is the witness real for {\sf HP}.

We still work in $L[B_0, A, C]$. Take  $\alpha$ and $X$ such that $L_{\alpha}[A]\models Z_3, X\prec L_{\alpha}[A,C]$, $|X|=\omega$ and $X\cap\omega_1\in Lim(C)$. Let $\eta=X\cap\omega_1$. The transitive collapse of $X$ is in the form $L_{\bar{\alpha}}[A\cap \eta,C\cap \eta]$. Note that $L_{\bar{\alpha}}[A\cap \eta]\models Z_3$ and
\begin{equation}\label{new equ today}
\text{$L_{\bar{\alpha}}[A\cap \eta,C\cap \eta]\models \eta=\omega_1$.}
\end{equation}
By (\ref{new equ today}), $L_{\bar{\alpha}}[A\cap \eta]\models \eta=\omega_1$. So $\alpha_{\eta}\leq\bar{\alpha}$. By (\ref{new equ today}), $L_{\alpha_{\eta}}[A\cap \eta,C\cap \eta]\models \eta=\omega_1$. Since $\eta\in Lim(C)$, by (\ref{thrd pro of c}), $L_{\alpha_{\eta}}[A\cap \eta,C\cap \eta]\models Z_3$. Let
\begin{equation}\label{definition of eta}
\text{$\eta^{\ast}$ be the least $\eta\in Lim(C)$ such that $L_{\alpha_{\eta}}[A\cap \eta,C\cap \eta]\models Z_3\, + \,\eta=\omega_1$.}
\end{equation}
Note that $\eta^{\ast}$ is a limit point of $Lim(C)$.\footnote{Suppose not. Let $\xi<\eta^{\ast}$ be the largest element of $Lim(C)$. Then $o.t.(C\cap (\eta^{\ast}\setminus (\xi+1)))=\omega$. But since $L_{\alpha_{\eta^{\ast}}}[A\cap \eta^{\ast}, C\cap \eta^{\ast}]\models \eta^{\ast}=\omega_1, L_{\alpha_{\eta^{\ast}}}[A\cap\eta^{\ast}, C\cap\eta^{\ast}]\models C\cap\eta^{\ast}$  is a club in $\eta^{\ast}$. Contradiction.}

\begin{lemma}\label{Basic properties of D}
Suppose $\eta\in Lim(C), \eta<\eta^{\ast}$ and  $\beta<\alpha_{\eta}$. Then  $L_{\alpha_{\eta}}[A\cap\eta, C\cap\eta]\models \beta<\omega_1$.
\end{lemma}
\begin{proof}
Since $L_{\alpha_{\eta}}[A\cap\eta]\models Z_3, L_{\alpha_{\eta}}[A\cap\eta]\models \forall \beta\in Ord(|\beta|\leq \omega_1)$. Since $L_{\alpha_{\eta}}[A\cap\eta]\models \eta=\omega_1$ and $\beta<\alpha_{\eta}$, there exists $f\in L_{\alpha_{\eta}}[A\cap\eta]$ such that $f: \eta\rightarrow\beta$ is surjective.

\begin{claim}
$L_{\alpha_{\eta}}[A\cap\eta, C\cap\eta]\models\eta<\omega_1$.
\end{claim}
\begin{proof}
Suppose $L_{\alpha_{\eta}}[A\cap\eta, C\cap\eta]\models\eta=\omega_1$. By (\ref{thrd pro of c}), $L_{\alpha_{\eta}}[A\cap\eta, C\cap\eta]\models Z_3$. By (\ref{definition of eta}), $\eta\geq\eta^{\ast}$. Contradiction.
\end{proof}

So there exists $g\in L_{\alpha_{\eta}}[A\cap\eta, C\cap\eta]$ such that $g: \omega\rightarrow \eta$ is surjective. So $f\circ g: \omega\rightarrow \beta$ is surjective and $f\circ g\in L_{\alpha_{\eta}}[A\cap\eta, C\cap\eta]$. Hence $L_{\alpha_{\eta}}[A\cap\eta, C\cap\eta]\models \beta<\omega_1$.
\end{proof}

Now we work in $L_{\alpha_{\eta^{\ast}}}[A\cap\eta^{\ast}, C\cap\eta^{\ast}]$. We first define an almost disjoint system $\langle\delta_{\beta}: \beta<\eta^{\ast}\rangle$ on $\omega$ and $B^{\ast}\subseteq \eta^{\ast}$. To define  $\langle\delta_{\beta}:\beta<\eta^{\ast}\rangle$ we first define $\langle f_{\beta}: \beta<\eta^{\ast}\rangle$ by induction on $\beta<\eta^{\ast}$.  Let $\langle f_{\beta}: \omega\rightarrow 1+\beta\mid  \beta<\omega\rangle$ be an uniformly defined sequence of recursive functions.\footnote{i.e. Take a recursive function $F:\omega\rightarrow \omega^{\omega}$ such that $F(\beta)(n)=f_{\beta}(n)$.} 

Fix $\omega\leq\beta<\eta^{\ast}$. Let  $\eta_0=\sup(Lim(C)\cap\beta)$ and $\eta_1=\min(C\setminus (\beta+1))$.
\begin{definition}\label{def of almost disjont syst}
\begin{enumerate}[(i)]
        \item Suppose $\eta_0=0$. Since $\eta_1\in C$ and $\beta<\eta_1$, by Lemma \ref{Basic properties of S}(a), $L_{\eta_1}[A]\models \beta$ is countable. Let $f_{\beta}:\omega\rightarrow\beta$ be the least surjection in $L_{\eta_1}[A]$.
        \item Suppose $\eta_0\neq 0$ and $\beta<\alpha_{\eta_0}$. Since $\eta_0\in Lim(C), \eta_0<\eta^{\ast}$ and $\beta<\alpha_{\eta_0}$, by Lemma \ref{Basic properties of D}, $L_{\alpha_{\eta_0}}[A\cap\eta_0, C\cap\eta_0]\models \beta<\omega_1$. Let $f_{\beta}: \omega\rightarrow\beta$ be the least surjection in $L_{\alpha_{\eta_0}}[A\cap \eta_0, C\cap  \eta_0]$.
            \item Suppose $\eta_0\neq 0$ and $\beta\geq\alpha_{\eta_0}$. Since $\eta_1\in S$ and $\beta<\eta_1$, by Lemma \ref{Basic properties of S}(a), $L_{\eta_1}[A]\models \beta$ is countable. Let $f_{\beta}: \omega\rightarrow\beta$ be the least surjection in $L_{\eta_1}[A]$.
      \end{enumerate}
\end{definition}

Now we define an almost disjoint system $\langle\delta_{\beta}: \beta<\eta^{\ast}\rangle$ on $\omega$ from $\langle f_{\beta}: \beta<\eta^{\ast}\rangle$. Fix a recursive bijection $\pi: \omega\leftrightarrow\omega\times\omega$. Let $x_{\beta}=\{(i,j)\mid  f_{\beta}(i)<f_{\beta}(j) \}$ and $y_{\beta}=\{k\in\omega\mid \pi(k)\in x_{\beta}\}$. Let $\langle s_i\mid i\in\omega\rangle$ be an injective, recursive enumeration of $\omega^{<\omega}$ and $\delta_{\beta}=\{i\in\omega\mid \exists m\in\omega(s_i=y_{\beta}\cap m)\}$. Then $\langle\delta_{\beta}: \beta<\eta^{\ast}\rangle$ is a sequence of almost disjoint reals. Since $\langle s_i\mid i\in\omega\rangle$ is recursive, $\pi$ is recursive and for any $i\in\omega, f_i$ is recursive, $\langle\delta_{i}: i\in\omega\rangle$ is recursive.

Now we define $B^{\ast}\subseteq \eta^{\ast}$. Fix $\beta<\eta^{\ast}$. We define $z_{\beta}$ as follows. Let
\begin{equation}\label{def of two next num}
\text{$\eta^{\beta}_0=\min(Lim(C)\setminus (\beta+1))$ and $\eta^{\beta}_1=\min(Lim(C)\setminus (\eta^{\beta}_0+1))$. }
\end{equation}
Note that $\eta^{\beta}_1<\eta^{\ast}$ since $\eta^{\ast}$ is a limit point of $Lim(C)$. By Lemma \ref{Basic properties of S}(b), $\alpha_{\eta^{\beta}_0}<\alpha_{\eta^{\beta}_1}$. By Lemma \ref{Basic properties of D}, $\alpha_{\eta^{\beta}_0}$ is countable in $L_{\alpha_{\eta^{\beta}_1}}[A\cap\eta^{\beta}_1, C\cap\eta^{\beta}_1]$. Let $z_{\beta}$ be the least real in $L_{\alpha_{\eta^{\beta}_1}}[A\cap\eta^{\beta}_1, C\cap\eta^{\beta}_1]$ such that
\begin{equation}\label{def of least real}
\text{$z_{\beta}$ codes $\langle \eta^{\beta}_0,  \alpha_{\eta^{\beta}_0}, A\cap\eta^{\beta}_0,  C\cap\eta^{\beta}_0\rangle$.}
\end{equation}\begin{equation}\label{def of b new}
\text{Define $B^{\ast}=\{\omega\cdot\alpha+i\mid \alpha<\eta^{\ast}\wedge i\in z_{\alpha}\}$.}
\end{equation}

By almost disjoint forcing, we get a  real $x$ such that for $\alpha<\eta^{\ast}$,
\begin{equation}\label{result of adf}
\text{$\alpha\in B^{\ast}\Leftrightarrow |x\cap \delta_{\alpha}|<\omega$. }
\end{equation}
Since $L_{\alpha_{\eta^{\ast}}}[A\cap\eta^{\ast},C\cap\eta^{\ast}]\models Z_3$ and $x$ is a generic real built via a $c.c.c$ forcing, by Fact \ref{forcing preserving system lemma}, $L_{\alpha_{\eta^{\ast}}}[A\cap\eta^{\ast},C\cap\eta^{\ast}][x]\models Z_3$. By (\ref{result of adf}), (\ref{def of b new}) and (\ref{def of least real}), $x$ codes $(A\cap\eta^{\ast}, C\cap\eta^{\ast})$ via $\langle\delta_{\beta}: \beta<\eta^{\ast}\rangle$. 

We want to show that $L_{\alpha_{\eta^{\ast}}}[A\cap\eta^{\ast},C\cap\eta^{\ast}][x]\models$ {\sf HP}. By absoluteness,  it suffices to show in $L[B_0, A, C, x]$ that if $\lambda<\alpha_{\eta^{\ast}}$ is $x$-admissible, then $\lambda$ is an $L$-cardinal. Now we work in $L[B_0, A, C, x]$. In the rest of this section, we fix $\lambda<\alpha_{\eta^{\ast}}$ and assume that
\begin{equation}\label{fact on admissibility}
\text{$\lambda$ is $x$-admissible.}
\end{equation}

Since $\langle\delta_{i}\mid i\in\omega\rangle$ is recursive, by (\ref{fact on admissibility}), $\langle\delta_{i}\mid i\in\omega\rangle\in L_{\lambda}[x]$. By (\ref{def of b new}), $B^{\ast}\cap \omega=z_0$. By (\ref{result of adf}), $B^{\ast}\cap \omega=\{i\in\omega\mid |x\cap \delta_{i}|<\omega\}$. By (\ref{fact on admissibility}),  $z_0\in L_{\lambda}[x]$.

\begin{definition}
$\theta=\sup(\{\beta<\eta^{\ast}\mid z_{\beta}\in L_{\lambda}[x]\})$ and  $\gamma=\sup(\{\eta^{\beta}_0\mid \beta<\theta\})$.
\end{definition}

By (\ref{def of two next num}) and (\ref{def of least real}), for $\beta<\eta^{\ast}$, $z_{\beta}=z_{\beta+1}$. So $\theta$ is a limit ordinal. By (\ref{def of least real}), if $\beta_0<\beta_1<\eta^{\ast}$, then $z_{\beta_0}$ is recursive in $z_{\beta_1}$. So if  $\beta<\theta$, then by (\ref{fact on admissibility}), $z_{\beta}\in L_{\lambda}[x]$. Note that  $z_{\beta}$ codes $(A\cap\eta^{\beta}_0,  C\cap\eta^{\beta}_0)$ for $\beta<\theta$ and hence $(A\cap\gamma,  C\cap\gamma)\in L_{\lambda}[x]$.

\begin{lemma}\label{definability lemma}
Suppose $\theta<\lambda$. Then $\langle z_{\beta}\mid\beta<\theta\rangle$ is $\Sigma_1$-definable in $L_{\lambda}[x]$ from $(A\cap\gamma,  C\cap\gamma)$.
\end{lemma}
\begin{proof}
If $\beta<\theta$, then $z_{\beta}\in L_{\lambda}[x]$ and hence by (\ref{def of least real})  and (\ref{fact on admissibility}), there exists $\lambda_0<\lambda$ such that  $\lambda_0$ is a limit ordinal and $\langle \eta^{\beta}_0,  \alpha_{\eta^{\beta}_0}, A\cap\eta^{\beta}_0,  C\cap\eta^{\beta}_0\rangle\in L_{\lambda_0}[x]$. We can find a formula $\varphi(\alpha,z,\beta,x, A\cap\gamma,C\cap\gamma)$ which says that $\langle \eta^{\beta}_0,  \alpha_{\eta^{\beta}_0}, A\cap\eta^{\beta}_0,  C\cap\eta^{\beta}_0\rangle$ is countable in  $L_{\alpha}[x]$ and $z$ is the $<_{L_{\alpha}[x]}$-least  real which codes $\langle \eta^{\beta}_0,  \alpha_{\eta^{\beta}_0}, A\cap\eta^{\beta}_0,  C\cap\eta^{\beta}_0\rangle$. By absoluteness, for $\beta<\theta, z=z_{\beta}$ if and only if $\exists\lambda_0<\lambda(z\in L_{\lambda_0}[x]\wedge \lambda_0$ is a limit ordinal $\wedge L_{\lambda_0}[x]\models \varphi[\lambda_0, z,\beta,x, A\cap\gamma,C\cap\gamma])$.
\end{proof}

\begin{theorem}\label{proof of the main thesis theorem}
$\lambda$ is an $L$-cardinal.
\end{theorem}
\begin{proof}
If $\beta<\theta$, then since $z_{\beta}$ codes $\eta^{\beta}_0$ and $z_{\beta}\in L_{\lambda}[x]$, by (\ref{fact on admissibility}),  $\beta<\eta^{\beta}_0<\lambda$. Hence $\theta\leq\lambda$ and $\gamma\leq\lambda$.

Case 1: $\theta=\lambda$. Then $\gamma=\sup(\{\eta^{\beta}_0\mid \beta<\lambda\})=\lambda$. Since $\gamma\in Lim(C)$, by (\ref{first pro of c}), $\lambda$ is an $L$-cardinal.

Case 2: $\theta<\lambda$. Since $(A\cap\gamma,  C\cap\gamma)\in L_{\lambda}[x]$, by Lemma \ref{definability lemma} and (\ref{fact on admissibility}), $\langle z_{\beta}\mid\beta<\theta\rangle\in L_{\lambda}[x]$. 

Subcase 1: $\alpha_{\gamma}\leq\lambda$. Since $\gamma, \eta^{\ast}\in Lim(C)$ and $\lambda<\alpha_{\eta^{\ast}}$, by Lemma \ref{Basic properties of S}(b), $\gamma<\eta^{\ast}$. For $i\in\omega$, since $\gamma+i<\alpha_{\gamma}$, by Definition \ref{def of almost disjont syst}(ii),  $f_{\gamma+i}: \omega\rightarrow \gamma+i$ is the least surjection in $L_{\alpha_{\gamma}}[A\cap\gamma, C\cap\gamma]$.\footnote{This is the place we use (\ref{thrd pro of c}): Definition \ref{def of almost disjont syst}(ii) uses Lemma \ref{Basic properties of D} which follows from (\ref{thrd pro of c}).} So $\langle \delta_{\gamma+i}\mid i\in\omega\rangle$ is $\Sigma_1$-definable in $L_{\alpha_{\gamma}}[A\cap\gamma, C\cap\gamma]$ from $(A\cap\gamma, C\cap\gamma)$. Since $\langle \delta_{\gamma+i}\mid i\in\omega\rangle$ is $\Sigma_1$-definable in $L_{\lambda}[x]$ from $(A\cap\gamma, C\cap\gamma)$ and $(A\cap\gamma,  C\cap\gamma)\in L_{\lambda}[x]$, by (\ref{fact on admissibility}), $\langle \delta_{\gamma+i}\mid i\in\omega\rangle\in L_{\lambda}[x]$. Note that $\omega\cdot\gamma=\gamma$ and $z_{\gamma}=\{i\in\omega\mid \omega\cdot\gamma+i\in B^{\ast}\}=\{i\in\omega\mid |x\cap\delta_{\gamma+i}|<\omega\}$. By (\ref{fact on admissibility}), $z_{\gamma}\in L_{\lambda}[x]$ and hence $\gamma<\theta$. By the definition of $\gamma, \eta^{\gamma}_0\leq \gamma$. Contradiction.

Subcase 2: $\lambda<\alpha_{\gamma}$. Since $A\cap\gamma\in L_{\lambda}[x]$, by (\ref{fact on admissibility}),  $\lambda$ is $A\cap\gamma$-admissible. Since $\gamma \in Lim(C)$ and $\gamma\leq\lambda<\alpha_{\gamma}$, by (\ref{sec pro of c}), $\lambda$ is an $L$-cardinal.
\end{proof}

So $L_{\alpha_{\eta^{\ast}}}[A\cap\eta^{\ast},C\cap\eta^{\ast}][x]\models Z_3\, + \, {\sf HP}$ and we have proved  The Main Theorem \ref{the main thm of the thesis}.\footnote{To define an almost disjoint system on $\omega$, we usually use the reshaping technique. However, in our proof we did not use reshaping and instead we use properties of $Lim(C)$ to define the almost disjoint system.} As a corollary, $Z_3\, + \,  {\sf HP}$ does not imply $0^{\sharp}$ exists.\footnote{From \cite{Schindler2}, any remarkable cardinal is remarkable in $L$.}

\subsection{Conclusion}\label{final section} 

We give an outline of our proof of The Main Theorem \ref{the main thm of the thesis}. In Step One, we force over $L$ to get a club in $\omega_2$ of $L$-cardinals with the strong reflecting property. This is necessary to show in Step Two that (\ref{key resut in section two}) holds. In Step Two, we find some $B_0\subseteq\omega_2$ and $A\subseteq\omega_1$ such that (\ref{key resut in section two}) holds in $L[B_0, A]$. (\ref{key resut in section two}) motivates the  definition of $S$ and is necessary to show that $S$ as defined in (\ref{def of S}) contains a club in $\omega_1$ and hence is stationary. In Step Three, we shoot a club  $C$ through $S$ via Baumgartner's forcing such that (\ref{thrd pro of c}) holds.  (\ref{thrd pro of c}) will be used to define the almost disjoint system and show that the generic real via almost disjoint forcing satisfies {\sf HP}. In Step Four, we use properties of $Lim(C)$(Lemma \ref{Basic properties of S} and Lemma \ref{Basic properties of D}) to define the almost disjoint system  on $\omega$ and some $B^{\ast}\subseteq \omega_1$. Then we do almost disjoint forcing to code $B^{\ast}$ by a real $x$.  Finally, we use properties of $Lim(C)$((\ref{first pro of c}), (\ref{sec pro of c}) and (\ref{thrd pro of c})) to show that $x$ is the witness real for {\sf HP}. 

From the proof of The Main Theorem \ref{the main thm of the thesis}, if we can force  a club in $\omega_2$ of $L$-cardinals with the weakly reflecting property via set forcing, then we can force  a set model of $Z_3\, + \, {\sf HP}$ via set forcing without reshaping. In our proof, the hypothesis ``there exists a remarkable cardinal with a weakly inaccessible cardinal above it" is  only used in Step One to force a club in $\omega_2$ of $L$-cardinals with the weakly reflecting property. 

We give a remark about the amount of the strong reflecting property  needed in our proof. For our proof, we need that $\omega_2$ has the strong reflecting property. Only knowing that some $\gamma\in [\omega_1, \omega_2)$ has the strong reflecting property is not enough for our proof. From this observation, only assuming one remarkable cardinal is not enough for our proof.

\begin{acknowledgement}
This paper is based on part of the author's
Ph.D.\ thesis written in 2012 at the National University of Singapore under the supervision of Chong Chi Tat and W.Hugh Woodin. I would like to express my deep gratitude to W.Hugh Woodin for all his support and guidance on the thesis. I would like to thank members of my Ph.D. committee. I would like to thank Ralf Schindler for  pointing out that Step One can be done by assuming a weaker large cardinal: a remarkable cardinal. I would like to thank  referees for their careful reading and helpful comments. 
\end{acknowledgement}

\end{document}